\documentclass[10pt, reqno]{amsart}
\usepackage{amssymb}
\usepackage{amsmath,amscd}
\usepackage[initials,nobysame,alphabetic]{amsrefs}
\usepackage{amsmath, amssymb}
\usepackage{amsfonts}
\usepackage{mathrsfs}
\usepackage{mathpazo}
\usepackage[arrow,matrix,curve,cmtip,ps]{xy}

\usepackage{amsthm}

\usepackage{float}
\usepackage{hyperref}
\usepackage{graphics}
\usepackage{graphicx}
\usepackage{verbatim}
\usepackage{multirow}
\usepackage{tikz}                                                                 
\usepackage{enumerate}
\allowdisplaybreaks

\newtheorem{theorem}{Theorem}[section]

\newtheorem{proposition}[theorem]{Proposition}

\theoremstyle{remark}
\newtheorem{remark}[theorem]{Remark}

\newtheorem{example}[theorem]{Example}

\numberwithin{equation}{section}

\newcommand{\bal}{\begin{align}}
\newcommand{\eal}{\end{align}}

\newcommand{\E}{\mathbb{E}}
\newcommand{\R}{\mathbb{R}}

\newcommand{\U}{\mathcal{U}}
\newcommand{\M}{\mathcal{M}}

\newcommand{\F}{\mathcal{F}}
\def \cP {\mathcal{ P}}
\def \bP {\mathbb{ P}}
\def\bF {\mathbb{ F}}
\def\cF {\mathcal{ F}}
\def \bM {\mathbb{ M}}
\def \fr {\forall}
\def \s {\sigma}
\def \le {\leq}

\def\as {\mbox{-a.s.}}
\def \ms {\medskip}
\def \bs {\bigskip}
\def \nd {\noindent}

\def \phi {\varphi}
\def \nn {\nonumber}

\def \lb {\label}
\def \D {\Delta}

\def \tx{\text}
\def \ed {\end{document}}
\def \G {\Gamma}
\def \ms{\medskip}
\def \bs{\bigskip}
\def \qd {\qquad}

\def \sdm {{\mathcal S}^{2,m}}
\def \hdm {{\mathcal H}^{2,m\times m}}

\newcommand{\Prob}{\mathbb{\Prob}}

\newcommand{\mytilde}{\raise.17ex\hbox{$\scriptstyle\mathtt{\sim}$}}

\usepackage[margin=1.3in]{geometry}

\begin{document}
\title[Mean-field backward-forward SDEs and  stochastic differential games]{Mean-field backward-forward stochastic differential equations and  nonzero sum stochastic differential games}

\author{Yinggu Chen, Boualem Djehiche and Said Hamad\`ene}

\address{Department of Mathematics \\ Shandong University\\ Jinan, Shandong Province \\ China}
\email{272564198@qq.com}
\address{Department of Mathematics \\ KTH Royal Institute of Technology \\ 100 44, Stockholm \\ Sweden}
\email{boualem@kth.se}
\address{Le Mans University, LMM \\ Avenue Olivier Messiaen \\ 72085 Le Mans, Cedex 9, France}
\email{ hamadene@univ-lemans.fr}
\thanks{The second author gratefully acknowledges financial support (grant 2016-04086) from the Swedish Research Council }

\date{\today}



\begin{abstract} We study a general class of fully coupled backward-forward stochastic differential equations of mean-field type (MF-BFSDE). We derive existence and uniqueness results for such a system under weak monotonicity assumptions and without the non-degeneracy  condition on the forward equation. This is achieved by suggesting  an implicit approximation scheme that is shown to converge to the solution  of the system of MF-BFSDE. We apply these results to derive an explicit form of open-loop Nash equilibrium strategies for nonzero sum mean-field linear-quadratic stochastic differential games with random coefficients.  These strategies are valid for any time horizon of the game.
\end{abstract}

\maketitle

\tableofcontents
\section{Introduction}

We study the solvability of the following backward-forward stochastic differential equation of mean-field type (MF-BFSDE): for every $t\le T$, 
\begin{equation}\label{mf-bfsde-intro}
\left\{\begin{array}{lll}
X_t=x+\int_0^t f(s,X_s,Y_s,Z_s,\mathbb{P}_{(X_s, Y_s)})ds+\int_0^t \sigma (s,X_s,Y_s,Z_s,\mathbb{P}_{(X_s, Y_s)})dW_s,\\ \\
Y_t=g(X_T,\mathbb{P}_{X_T})-\int_t^Th(s,X_s,Y_s,Z_s,\mathbb{P}_{(X_s, Y_s)})ds-\int_t^T Z_sdW_s,
\end{array}
\right.
\end{equation}
where $W:=(W_t)_{t\le T}$ is a standard Brownian motion on $\R^m$ defined on a probability space $(\Omega,\F,\bP)$, $\bP_{(X_t, Y_t)}$ is the $t$-marginal distribution of $(X_t,Y_t)$ and $f, h, \sigma$ and $g$ are Lipschitz continuous functions with appropriate dimensions.
\medskip

\noindent This class of MF-BSDEs appears  in the analysis of optimal control problems (the stochastic maximum principle) and nonzero-sum games related to nonlinear stochastic dynamical systems of McKean-Vlasov type (see e.g. \cite{AD, BDL, BLM, carmona2013, carmona2015, DH}, the list of related papers being far longer).  It is an extension of the standard BFSDEs studied in several papers including \cite{antonelli, H98, hu-yong, hu-peng95, MaP, Min, Peng-wu, hu1995solution}.

Under Lipschitz continuity and monotonicity conditions on the coefficients we derive existence and uniqueness results  for the system \eqref{mf-bfsde-intro}. Compared with e.g. \cite{carmona2013}, we do not require non-degeneracy of the diffusion coefficient of the forward process. We further allow it to depend on $Z$. 

The monotonicity condition appears first in the paper by Hu and Peng \cite{hu-peng95} in order to remedy the assumption related to the length $T$ of the horizon $[0,T]$ when dealing with the existence and uniqueness of the solution to the standard backward-forward SDE (equation \eqref{mf-bfsde-intro} when the coefficients $f$, $g$, $h$ and $\sigma$ do not depend on $\nu$). See also \cite{antonelli} for more details. Subsequent papers on the solvability of standard BFSDEs  where the  monotonicity condition is substantially weakened include \cite{H98,Peng-wu}.

As mentioned above, when the data $f$, $g$, $h$ and $\sigma$  do not depend on $\nu$, the monotonicity condition is sufficient to obtain existence and uniqueness of a solution to the standard BFSDE. Therefore, an important issue is, beside the monotonicity condition, what kind of assumptions should be further imposed on the data (especially w.r.t. $\nu$) in order to obtain existence and uniqueness of a solution to the mean-field BFSDE \eqref{mf-bfsde-intro}. We show that if the Lipschitz constants of $f$, $g$, $h$ and $\sigma$ w.r.t. $\nu$ are small enough then \eqref{mf-bfsde-intro} has a unique solution.  When $\sigma$ does not depend on $\nu$,  we give a refinement of that result, under a relaxed monotonicity condition. This feature on $\sigma$ appears in the study of some linear-quadratic nonzero-sum differential games, which we consider in the second part of this paper. 

In the second part of the paper we deal with the linear-quadratic nonzero-sum differential game. The coefficients are stochastic processes and not necessarily deterministic. By using the stochastic maximum principle for optimal stochastic control problems obtained in \cite{AD}, we reduce the problem of existence of Nash equilibrium point (NEP for short) of the game to the solution of an associated MF-BFSDE of the type considered in the first part. Then, we provide conditions on the data of the game under which this latter MF-BFSDE has a unique solution and, consequently, the game has a NEP for any horizon $T$ whose explicit expression is also given. To the best our knowledge, this result seems new. Finally, there are some (rather few) other papers  on linear-quadratic nonzero-sum differential games including \cite{duncan-tembine,MP}, whose frameworks are however different from ours.  Actually, in \cite{duncan-tembine}, the main tool is the square completion technique, and in \cite{MP}, the method is based on the resolution of the associated Riccati equation. 

The paper is organized as follows. In Section 2, we formulate the problem and present our main results about existence and uniqueness of solutions to two classes of MF-BFSDEs under two different sets of monotonicity conditions, (H1) and (H$1^{\prime}$). In Section 3, we derive necessary and sufficient conditions for the existence of open loop Nash equilibrium strategies for an $n$-players nonzero sum mean-field linear-quadratic SDEs  with random coefficients. Moreover, we give an explicit form of those strategies. Finally, we give a counterexample to show that a Nash equilibrium may not exist when the monotonicity condition on the coefficients is not satisfied. 
    

\section{Mean-field Backward-forward stochastic differential equations}
Before we describe the framework defining the system of backward-forward SDEs, we introduce the Wasserstein distance between two probability measures. Denote by $\bM_2(\R^k)$ the set of probability measures on $\R^k$ with finite moments of order 2.  For $\mu_1,\mu_2\in\bM_2(\R^k)$, the 2-Wasserstein distance is defined by the formula
\begin{equation}\label{W-2}
d(\mu_1,\mu_2):=\inf\left\{\left(\int_{\R^k\times \R^k}|x-y|^2 F(dx,dy)\right)^{1/2};\,F(.,\R^k)=\mu_1,\,F(\R^k,.)=\mu_2\right\}
\end{equation}
i.e., the infimum is taken over $F\in\bM_2(\R^k\times \R^k)$ with marginals $\mu_1$ and $\mu_2$. It has also the following formulation in terms of a coupling between two square-integrable random variables $\xi$ and $\xi^{\prime}$ defined on the same probability space $(\Omega,\cF,\bP)$:
\begin{equation}\label{d-coupling}
d(\mu,\nu)=\inf\left\{\left(\E\left[|\xi-\xi^{\prime}|^2\right]\right)^{1/2},\,\,\text{law}(\xi)=\mu_1,\,\text{law}(\xi^{\prime})=\mu_2 \right\},
\end{equation}
from which is derived the following inequality involving the Wasserstein metric between the laws of the square integrable random variables $\xi, \bar\xi$ and their $L^2$-distance: 
\begin{equation}\label{dist}
d^2(\bP_{\xi},\bP_{\bar\xi},)  \le \E[|\xi-\bar\xi|^2],
\end{equation}
where $\bP_{\xi}:=\text{law}(\xi)$ and $\bP_{\xi^{\prime}}:=\text{law}(\xi^{\prime})$.

\medskip
Next let $(W_t)_{0\le t\le T}$  denote a standard $m$-dimensional Brownian motion, defined on the probability space $(\Omega,\cF,\bP)$, whose natural filtration is $(\cF^0_t)_{0\le t\le T}$, where $\cF^0_t=\sigma(W_s,\, s\le t)$ and we denote by 
$\mathbb{F}:=(\cF_t)_{0\le t\le T}$ its completion with the $\bP$-null sets of $\cF$. Let $\mathcal{P}$ be the $\sigma$-algebra of $\bF$-progressively measurable sets on $[0,T]\times \Omega$.  Set $\R^{m+m+m\times m}:=\R^m\times\R^m\times L(\R^m;\R^m)$ and let $\mathcal{M}^{2,k}$ denote  the space of $\mathcal{P}$-measurable and $\R^k$-valued processes which belong to $L^2([0,T]\times\Omega,dt\otimes d\bP)$. Next, we introduce the following spaces:

\begin{itemize}
\item[(i)] $\sdm$ is the space of continuous $\mathcal P$-measurable 
$\R^m$-valued processes $\zeta:=(\zeta_t)_{t\le T}$ such that \\ $\E[\sup_{t\le T}|\zeta_t|^2]<\infty$ ;

\item[(ii)] $\hdm$ is the space of $\mathcal P$-measurable 
$\R^{m\times m}$-valued processes $\theta:=(\theta_t)_{t\le T}$ such that \\$\E[\int_0^T|\theta_t|^2]<\infty$.
\end{itemize}

For $x, y\in \R^m, x\cdot y$ denotes the scalar product and for any $A, B\in L(\R^m,\R^d)$, $[A,B]=\sum_{j=1}^ d A^j.B^j$, $\, A^j, B^j$ being the $j$th columns of $A$ and $B$, respectively. Furthermore, for $u=(x,y,z)\in \R^{m+m+m\times m}$, we set $\|u\|^2:=|x|^2+|y|^2+\|z\|^2$, where $\|z\|^2=\mbox{trace}(zz^\top)$; ($^\top$) is the transpose operation.

\medskip
We make the following assumptions.
\begin{enumerate}
\item $g$ is a function defined on $\Omega\times\R^m\times \bM_2(\R^m)$ and valued in $\R^m$ such that, 
\begin{itemize}
\item[(a)] for any $(x,\mu)\in \R^m\times\bM_2(\R^m), \, g(x,\mu)$ is $\F_T^0$-measurable and square-integrable; \\
\item[(b)] $g$ is Lipschitz in $(x,\mu)$ uniformly in $\omega\in \Omega$, i.e. there exists positive constants $C_g^\nu$ and $C_g^x $ such that, for any $x,x^{\prime}\in \R^m$ and any $\nu, \nu^{\prime}\in \bM_2(\R^m)$,
\begin{align}\lb{lipschitzg}|g(x,\mu)-g(x',\mu^{\prime})|\le C_g^x \,|x-x^{\prime}|+ C_g^\nu d(\mu,\mu^{\prime}), \quad \bP\as
\end{align}
\end{itemize}

\item $f, h$ and $\sigma$ are functions defined on $[0,T]\times \Omega\times\R^{m+m+m\times m}\times \bM_2(\R^{m}\times\R^m)$, valued respectively in $\R^m,\R^m$ and $L(\R^m; \R^m)$ and satisfy
\begin{itemize}
\item[(a)] For any $\nu\in\bM_2(\R^{m}\times\R^m), u=(x,y,z)\in \R^{m+m+m\times m}$, the processes $(f(t,u,\nu))_{0\le t\le T}$, $(h(t,u,\nu))_{0\le t\le T}$ and $(\sigma(t,u,\nu))_{0\le t\le T}$ belong respectively to $\mathcal{M}^{2,m}, \mathcal{M}^{2,m}$ and $\mathcal{M}^{2,m\times m}$. \\

\item[(b)] $f,h$ and $\sigma$ are Lipschitz in $(x,y,z,\nu)$ uniformly in $(t,\omega)\in [0,T]\times \Omega$, i.e. for $\phi=f,b,\sigma$, there exist positive constants $C_\phi^{\nu} $ and $ C_\phi^u$ such that for any $t\in [0,T], u=(x,y,z), u^{\prime}=(x^{\prime},y^{\prime}, z^{\prime})\in \R^{m+m+m\times m}$, $\nu,\nu^{\prime}\in \bM_2(\R^{m}\times\R^m)$
\begin{equation}\label{lipschitz}\begin{array}{lll}
|\phi(t,u,\nu)-\phi(t,u^{\prime},\nu^{\prime})| \le C_\phi^u\, \|u-u^{\prime}\|+ C_\phi^{\nu}\, d(\nu,\nu^{\prime}),\quad \bP\as
\end{array}
\end{equation}

\end{itemize}

\end{enumerate}
Hereafter, we will use $C^u:=\max(C^u_f,C^u_h,C^u_{\sigma})$ and 
$C^{\nu}:=\max(C^{\nu}_f,C^{\nu}_h,C^{\nu}_{\sigma})$ as common Lipschitz constants of $f,h$ and $\sigma$ w.r.t. $u$ and $\nu$, respectively.  

\medskip 
A solution to the backward-forward stochastic differential equation associated with $(f,\sigma,h,g)$ is a triple of processes $(X,Y,Z):=(X_t,Y_t,Z_t)_{t\le T}$ which is $\R^{m+m+m\times m}$-valued such that 
\begin{equation}\label{mf-bfsde}
\left\{\begin{array}{l}
X,Y\in \sdm \,, \, Z\in \hdm\,;\\\\
X_t=x+\int_0^t f(s,X_s,Y_s,Z_s,\mathbb{P}_{(X_s, Y_s)})ds+\int_0^t \sigma (s,X_s,Y_s,Z_s,\mathbb{P}_{(X_s, Y_s)})dW_s,\,\,\,t\le T ;\\ \\
Y_t=g(X_T,\mathbb{P}_{X_T})-\int_t^Th(s,X_s,Y_s,Z_s,\mathbb{P}_{(X_s, Y_s)})ds-\int_t^T Z_sdW_s,\,\,\, t\le T.
\end{array}
\right.
\end{equation}

Next, for $t\in [0,T], \nu\in\bM_2(\R^{m}\times\R^m), u=(x,y,z)$ and $u^{\prime}=(x^{\prime},y^{\prime},z^{\prime})$ in $\R^{m+m+m\times m}$, we define the function $\mathcal{A}$ by

\begin{equation}\lb{operateur}
\begin{array}{lll}
\mathcal{A}(t,u,u^{\prime},\nu):=(f(s,x,y,z,\nu)-f(s,x^{\prime},y^{\prime},z^{\prime},\nu))\cdot (y-y^{\prime}) \\ 
\qquad\qquad\qquad\quad+ (h(s,x,y,z,\nu)-h(s,x^{\prime},y^{\prime},z^{\prime},\nu))\cdot(x-x^{\prime})\\
\qquad\qquad\qquad\quad +[\sigma(s,x,y,z,\nu)-\sigma(s,x^{\prime},y^{\prime},z^{\prime},\nu),z-z^{\prime}].
\end{array}
\end{equation}

\medskip
We consider the following assumption. 
\begin{eqnarray*}
(H1)\left\{ \begin{array}{lll}
\text{(i) there exists $k >0$, s.t. for all } t\in[0,T], \nu\in\bM_2(\R^{m}\times\R^m), u,u^{\prime} \in \R^{m+m+m\times m}, \\
\quad \mathcal{A}(t,u,u^{\prime},\nu)\le -k(|x-x^{\prime}|^2+|y-y^{\prime}|^2+|z-z^{\prime}|^2),\quad \bP\as \\  \\
\text{(ii) there exists $k^{\prime} >0$, s.t. for all } \nu\in\bM_2(\R^{m}\times\R^m), x,x^{\prime}\in\R^m, \\
(g(x,\nu)-g(x^{\prime},\nu))\cdot(x-x^{\prime})\ge k^{\prime}|x-x^{\prime}|^2, \quad \bP\as  \\ 
\end{array}
\right.
\end{eqnarray*}
\begin{remark}In the case when $\mathcal A$ and $g$ do not depend on $\nu$, Assumption (H1) appear first in a paper by Hu-Peng \cite{hu-peng95} to study the existence and uniqueness of the solution of the backward-forward SDE \eqref{mf-bfsde-intro} in the framework where the coefficients do not depend on $\nu$. This assumption is then weakened in several papers including  
\cite{H98}, \cite{Peng-wu}.
\end{remark}

In the next section we prove existence and uniqueness of the solution of system \eqref{mf-bfsde} of backward-forward SDEs under the assumptions (H1).

\subsection{Existence and uniqueness results under (H1)}
\begin{theorem}[Existence and Uniqueness of a solution]\label{T123} 
Let assumption (H1) hold. If the constant $C_g^\nu, {C^\nu}$ satisfy the inequality \begin{equation}\label{coef condi}
C_g^\nu,C^{\nu}  < min \{(\sqrt{3}-1)k^{\prime}, \frac{\sqrt{3}}{3}k \}
\end{equation}
then there exists a unique process $U=(X,Y,Z)$ which solves the system \eqref{mf-bfsde} of Backward-Forward SDE of mean-field type.
\end{theorem}
\noindent \underline{Proof}. 
\noindent (i) \underline{Existence of a solution}: Let $\delta>0$ and consider the sequence $U^{n}=(X^n,Y^n,Z^n)_{n\ge 0}$ of processes defined recursively as follows: 
$(X^0,Y^0,Z^0)=(0,0,0)$ and, for $n \ge 0, U^{n+1}$ satisfies,  for every $0\le t\leq T$, 
\begin{equation}\label{T1-n}\left\{
\begin{array}{lll}
 U^{n+1}=(X^{n+1},Y^{n+1},Z^{n+1})\in \sdm\times \sdm\times \hdm \,\,; \\\\
X^{n+1}_t=x+\int_0^t \left\{f(s,U^{n+1}_s,\nu^{n}_s)-\delta (Y^{n+1}_s-Y^n_s)\right\}ds \\ \\ \qquad\qquad+\int_0^t \left\{\sigma (s,U^{n+1}_s,\nu^n_s)-\delta(Z^{n+1}_s-Z^n_s)\right\}dW_s,\\ \\
Y^{n+1}_t=g(X^{n+1}_T,\mu^n_T)-\int_t^Th(s,U^{n+1}_s,\nu^n_s)ds-\int_t^T Z^{n+1}_sdW_s,
\end{array}\right.
\end{equation}
where $\nu^n_t:=\mathbb{P}_{(X^n_t, Y^n_t)}$ and $\mu^n_T:=\mathbb{P}_{X^n_T}$. By Theorem 1.2 in \cite{H98} (see also \cite{hu1995solution}, pp.282 or \cite{Peng-wu}, pp.833), the system \eqref{T1-n} admits a unique solution. First we will show that $(U^n)_{n\ge 0}$ is a Cauchy sequence in $\mathcal{M}^{2,m+m+m\times m}$ and $(X^n_T)_{n\ge 0}$ is a Cauchy sequence in $L^2(d\bP)$.
For $n\ge 1,\ t\in[0,T]$, set 
\begin{equation}\lb{num1}
\hat{X}^{n+1}_t:=X^{n+1}_t-X^{n}_t,\,\,\, \hat{Y}^{n+1}_t:=Y^{n+1}_t-Y^{n}_t,\,\,\, \hat{Z}^{n+1}_t:=Z^{n+1}_t-Z^{n}_t
\end{equation}
and for $\varphi=f,h, \sigma$, 
\begin{equation}\lb{num2}\begin{array}{lll}
\widehat{\varphi}^{n+1}(t):=\varphi(t,U^{n+1}_t,\nu^{n}_t)-\varphi(t,U^{n}_t,\nu^{n-1}_t),\\ \overline{\varphi}^{n}(t):=\varphi(t,U^{n}_t,\nu^{n}_t)-\varphi(t,U^{n}_t,\nu^{n-1}_t).
\end{array}
\end{equation}
Applying It\^o's formula, we obtain
\begin{equation}\label{T1-duality-1}\begin{array}{lll}
\hat{X}^{n+1}_T\cdot\hat{Y}^{n+1}_T-\hat{X}^{n+1}_0\cdot\hat{Y}^{n+1}_0=\int_0^T \hat{Y}^{n+1}_s\cdot \{\hat{f}^{n+1}(s)-\delta(\hat{Y}^{n+1}_s-\hat{Y}^{n}_s)\}ds \\ \\ \qquad\qquad\qquad\qquad\qquad\qquad
+\int_0^T \hat{Y}^{n+1}_s\cdot \{\hat{\sigma}^{n+1}(s)-\delta(\hat{Z}^{n+1}_s-\hat{Z}^{n}_s)\}dW_s \\ \\ \qquad\qquad\qquad\qquad\qquad\qquad 
+\int_0^T \hat{X}^{n+1}_s\cdot \hat{h}^{n+1}(s)ds+ \int_0^T \hat{X}^{n+1}_s\cdot \hat{Z}^{n+1}_sdW_s   \\ \\ \qquad\qquad\qquad\qquad \qquad\qquad
+\int_0^T [\hat{\sigma}^{n+1}(s)-\delta(\hat{Z}^{n+1}_s-\hat{Z}^{n}_s), \hat{Z}^{n+1}_s]ds.
\end{array}
\end{equation}
Furthermore, using standard estimates of BSDEs and the Burkholder-Davis-Gundy inequality, it is easy to see that the stochastic integrals in \eqref{T1-duality-1} are true martingales. We may take expectation to obtain
\begin{eqnarray}\label{T1-duality-2}\begin{array}{lcl}
\E[\hat{X}^{n+1}_T\cdot (g(X^{n+1}_T,\mu_T^{n})- g(X^{n}_T,\mu_T^{n-1}))] +
\delta \E\left[\int_0^T \left( |\hat{Y}^{n+1}_s|^2+\| \hat{Z}^{n+1}_s\|^2\right)ds \right] \\\\ \qquad = \delta \E\left[\int_0^T\left(\hat{Y}^{n+1}_s\cdot \hat{Y}^{n}_s+[\hat{Z}^{n+1}_s,\hat{Z}^{n}_s]\right)ds\right]\\\\
\qquad \qquad +\E\left[\int_0^T\left( \hat{X}^{n+1}_s \cdot \hat{h}^{n+1}(s)+\hat{Y}^{n}_s\cdot\hat{f}^{n+1}(s)+[\hat{\sigma}^{n+1}(s),\hat{Z}^{n+1}_s]\right)ds\right].
 \end{array}
\end{eqnarray}
 \noindent Using the Lipschitz continuity of $g$,  Young's inequality, \eqref{dist} and (H1(ii)), we have, for any $\varepsilon>0$,
\begin{align}\label{eqintermed1}
&\E[\hat{X}^{n+1}_T\cdot (g(X^{n+1}_T,\mu_T^{n})- g(X^{n}_T,\mu_T^{n-1}))] \nn \\&=\E[\hat{X}^{n+1}_T\cdot (g(X^{n+1}_T,\mu_T^{n})- g(X^{n}_T,\mu_T^{n}))] +\E[\hat{X}^{n+1}_T\cdot (g(X^{n}_T,\mu_T^{n})- g(X^{n}_T,\mu_T^{n-1}))] \nn\\\nn\\ &
\ge k^{\prime}\E[|\hat{X}^{n+1}_T|^2]-C_g^\nu\E[|\hat{X}^{n+1}_T|]d(\mu_T^{n},\mu_T^{n-1}) \nn\\ \nn\\ & 
 \ge  k^{\prime}\E[|\hat{X}^{n+1}_T|^2]-\frac{C_g^\nu\varepsilon}{2} \E[|\hat{X}^{n+1}_T|^2]-\frac{C_g^\nu}{2\varepsilon}d^2(\mu_T^{n},\mu_T^{n-1}) \nn \\\nn \\ & \ge (k^{\prime}-\frac{C_g^\nu\varepsilon}{2})\E[|\hat{X}^{n+1}_T|^2]-\frac{C_g^\nu}{2\varepsilon}\E[|\hat{X}^{n}_T|^2].
\end{align}
Again, by the Lipschitz continuity of $f, h,\sigma$,  Young's inequality, \eqref{dist} and (H1(i)), we also have, for every $0\le  t\le T$ and any $\alpha>0$,
$$
 \begin{array}{lll}
&\hat{X}^{n+1}_t \cdot \hat{h}^{n+1}(t)+\hat{Y}^{n+1}_t\cdot \hat{f}^{n+1}(t)+[\hat{\sigma}^{n+1}(t),\hat{Z}^{n+1}_t]\\ \\&=\mathcal{A}(t,U_t^{n+1},U_t^{n},\nu^{n}_t)+\hat{X}^{n+1}_t \bar{h}^{n}(t)+\hat{Y}^{n+1}_t\cdot\bar{f}^{n}(t)+[\bar{\sigma}^{n}(t),\hat{Z}^{n+1}_t]\\  \\
&\le -k\left\{|\hat{X}^{n+1}_t|^2+|\hat{Y}^{n+1}_t|^2+\|\hat{Z}^{n+1}_t\|^2\right\}+|\hat{X}^{n+1}_t|| \bar{h}^{n}(t)|+|\hat{Y}^{n+1}_t||\bar{f}^{n}(t)|+|\bar{\sigma}^{n}(t)||\hat{Z}^{n+1}_t|\\ \\ &
\le -k\left\{|\hat{X}^{n+1}_t|^2+|\hat{Y}^{n+1}_t|^2+\|\hat{Z}^{n+1}_t\|^2\right\}+C^{\nu} d(\nu_t^{n},\nu_t^{n-1})\left(|\hat{X}^{n+1}_t|+|\hat{Y}^{n+1}_t|+|\hat{Z}^{n+1}_t|\right)\\ \\ &
\le -k\left\{|\hat{X}^{n+1}_t|^2+|\hat{Y}^{n+1}_t|^2+\|\hat{Z}^{n+1}_t\|^2\right\}+\frac{C^{\nu}}{2\alpha}\left\{|\hat{X}^{n+1}_t |^2+|\hat{Y}^{n+1}_t |^2+\|\hat{Z}^{n+1}_t \|^2\right\} +\frac{3\alpha C^{\nu}}{2} d^2(\nu_t^{n},\nu_t^{n-1}).\\ 
\end{array}
$$
Now since $d^2(\nu_t^{n},\nu_t^{n-1})\le \E[|\hat{X}^{n}_t|^2+|\hat{Y}^{n}_t|^2]$,  we have
 \begin{equation}\label{eqintermed2}
\begin{array}{lll}
\E\left[\int_0^T\left( \hat{X}^{n+1}_s \cdot\hat{h}^{n+1}(s)+\hat{Y}^{n+1}_s\cdot \hat{f}^{n+1}(s)+[\hat{\sigma}^{n+1}(s),\hat{Z}^{n+1}_s]\right)ds\right] \\ \\ \qquad \le 
\E\left[\int_0^T \left( (\frac{ C^{\nu}}{2\alpha}-k)\left\{|\hat{X}^{n+1}_t |^2+|\hat{Y}^{n+1}_t |^2+\|\hat{Z}^{n+1}_t \|^2\right\}+\frac{3\alpha C^{\nu}}{2}\E[|\hat{X}^{n}_t|^2+|\hat{Y}^{n}_t|^2]\right) ds\right].
\end{array}
\end{equation}
On the other hand, in view of Young's inequality, we also have, for any $\rho>0$, 
 \begin{equation}\label{eqintermed3}
\begin{array}{lll}
\E\left[\int_0^T\left(\hat{Y}^{n+1}_s\cdot \hat{Y}^{n}_s+[\hat{Z}^{n+1}_s,\hat{Z}^{n}_s]\right)ds\right]\le 
\frac{1}{2}\E\left[\int_0^T\left(\rho |\hat{Y}^{n+1}_s|^2  +\rho\|\hat{Z}^{n+1}_s\|^2 \right.\right.\\ \left.\left.  \qquad \qquad \qquad\qquad \qquad \qquad \qquad \qquad \qquad \qquad \qquad \qquad\qquad +\frac{1}{\rho}|\hat{Y}^{n}_s|^2+\frac{1}{\rho}\|\hat{Z}^{n}_s\|^2\right)ds\right].
\end{array}
 \end{equation}
Applying now \eqref{eqintermed1}, \eqref{eqintermed2} and \eqref{eqintermed3} to \eqref{T1-duality-2}, yields
\begin{equation*}\begin{array}{lll}
(k^{\prime}-\frac{C_g^\nu\varepsilon}{2})\E[|\hat{X}^{n+1}_T|^2]-\frac{C_g^\nu}{2\varepsilon}\E[|\hat{X}^{n}_T|^2] +
\delta \E\left[\int_0^T \left( |\hat{Y}^{n+1}_s|^2+\| \hat{Z}^{n+1}_s\|^2\right)ds \right]\\ \\ 
- \E\left[\int_0^T (\frac{ C^{\nu}}{2\alpha}-k)\left\{|\hat{X}^{n+1}_t |^2+|\hat{Y}^{n+1}_t |^2+\|\hat{Z}^{n+1}_t \|^2\right\}dt\right]
\\\\ \le \delta \E\left[\int_0^T\left(\frac{\rho}{2} |\hat{Y}^{n+1}_s|^2 +\frac{\rho}{2} \|\hat{Z}^{n+1}_s\|^2+\frac{1}{2\rho}|\hat{Y}^{n}_s|^2+\frac{1}{2\rho}\|\hat{Z}^{n}_s\|^2\right)ds\right]
+\frac{3\alpha C^{\nu}}{2}\E[\int_0^T\left\{|\hat{X}^{n}_t|^2+|\hat{Y}^{n}_t|^2\right\}dt].
\end{array}
\end{equation*}
Rearranging terms, we obtain
\begin{equation*}\begin{array}{lll}
&(k^{\prime}-\frac{C_g^\nu\varepsilon}{2})\E[|\hat{X}^{n+1}_T|^2]+
\E\left[\int_0^T(k-\frac{C^{\nu}}{2\alpha})|\hat{X}^{n+1}_t|^2ds\right] +\\ \\& \qquad\qquad\qquad 
\E\left[\int_0^T (\delta (1-\frac{\rho}{2} )+k-\frac{ C^{\nu}}{2\alpha})\left(|\hat{Y}^{n+1}_s|^2+\| \hat{Z}^{n+1}_s\|^2\right)ds \right] \\\\ & \le \frac{C_g^\nu}{2\varepsilon} \,\E[|\hat{X}^{n}_T|^2] +\E\left[\int_0^T\left(\frac{3\alpha C^{\nu}}{2}|\hat {X}^{n}_s|^2+(\frac{\delta}{2\rho}+\frac{3\alpha C^{\nu}}{2})|\hat{Y}^{n}_s|^2+\frac{\delta}{2\rho}\|\hat{Z}^{n}_s\|^2\right)ds\right].
\end{array}
\end{equation*}
By setting 
$$\begin{array}{l}
\lambda(\epsilon, \delta, \alpha,\rho):=\min\{k^{\prime}-\frac{C_g^\nu \varepsilon}{2}, k-\frac{C^{\nu}}{2\alpha},\delta (1-\frac{\rho}{2} )+k-\frac{C^{\nu}}{2\alpha}\}, \\
\theta(\epsilon, \delta, \alpha,\rho):=\max\{\frac{C_g^\nu}{2\varepsilon}, \frac{\delta}{2\rho}+\frac{3\alpha C^{\nu}}{2}\},
\end{array}$$
we obtain
\begin{equation}\label{T1-contraction}
\begin{array}{lll}
\E[|\hat{X}^{n+1}_T|^2]+\E\left[\int_0^T \| \hat{U}^{n+1}_s\|^2ds\right]\le \frac{\theta}{\lambda}\left(\E[|\hat{X}^{n}_T|^2]+\E\left[\int_0^T \|\hat{U}^{n}_s\|^2ds\right] \right).
\end{array}
\end{equation}
Now, if there exist $\alpha$, $\varepsilon$, $\delta$ and $\rho$ so that \begin{equation}\label{thetalambda}\lambda(\epsilon, \delta, \alpha,\rho)>\theta(\epsilon, \delta, \alpha,\rho)\end{equation}then the inequality \eqref{T1-contraction} becomes a contraction, which implies that $(X_T^n)_{n\ge 0}$ is a Cauchy sequence in $L^2(\Omega,\bP)$ and $(X^n)_{n\ge 0},(Y^n)_{n\ge 0}$ and $(Z^n)_{n\ge 0}$ are Cauchy sequences in  $L^2([0,T]\times\Omega, dt\otimes d\bP)$. Therefore going back to \eqref{T1-n}, using It\^o's formula and, by now standard calculations, we obtain 
$$\E[\sup_{s\le T}(|X^n_s-X^m_s|^2+|Y^n_s-Y^m_s|^2)]\rightarrow 0\mbox{ as }n,m\rightarrow\infty.$$
Consequently, there exist $\mathbb{F}$-adapted continuous processes $X$ and $Y$ and an $\mathbb{F}$-progressively measurable process $Z$ such that
$$\E[\sup_{s\le T}(|X^n_s-X_s|^2+|Y^n_s-Y_s|^2)+\int_0^T\|Z^n_s-Z_s\|^2ds]\rightarrow 0\,\,\, \mbox{ as }\,n\rightarrow\infty.$$
Moreover ,
$$\E[\sup_{s\le T}(|X_s|^2+|Y_s|^2)+\int_0^T\|Z_s\|^2ds]<\infty .$$
Finally, taking the limits in equation \eqref{T1-n} we obtain that $(X,Y,Z)$ is a solution of MF-BFSDE \eqref{mf-bfsde}. 

\medskip
Next, we are going to show that such  $\alpha$, $\varepsilon$, $\delta$ and $\rho$ exist when the condition \eqref{coef condi} is satisfied.
In fact, to make the contraction meaningful, we assume $k^{\prime}-\frac{C_g^\nu \varepsilon}{2}$, $k-\frac{C^{\nu}}{2\alpha}$ and $1-\frac{\rho}{2}$ are positive. It is easily shown that $(1-\frac{\rho}{2})\le\frac{1}{2\rho}$ and the terms of this inequality are equal if and only if $\rho=1$. So let us take $\rho=1$ and  set 
\begin{align*}&\theta^*(\epsilon,\alpha)=\lim_{\delta \rightarrow 0}\theta(\epsilon, \delta, \alpha,1)
=\max\{\frac{C_g^\nu}{2\varepsilon}, \frac{3\alpha C^{\nu}}{2}\}\\
&\mbox{and }
\lambda^*(\epsilon,\alpha)=\lim_{\delta \rightarrow 0}\lambda(\epsilon, \delta, \alpha,1)=\min\{k^{\prime}-\frac{C_g^\nu \varepsilon}{2}, k-\frac{C^{\nu}}{2\alpha}\}
.\end{align*}
Now if, for some $\epsilon$, $\alpha$, we have $\lambda^*(\epsilon,\alpha)>\theta^*(\epsilon,\alpha)$, then there exists $\delta$ small enough such \eqref{thetalambda} is satisfied with those $\epsilon$, $\alpha$, $\delta$ and $\rho=1$. Finally, in order to have $\lambda^*(\epsilon,\alpha)>\theta^*(\epsilon,\alpha)$, it is equal to have the following inequalities: 
\begin{equation}\label{contraction coeffcient 1}
\left\{\begin{aligned}
k^{\prime}-\frac{C_g^\nu \varepsilon}{2}& >\frac{C_g^\nu}{2\varepsilon} \\ k-\frac{C^{\nu}}{2\alpha}& >\frac{C_g^\nu}{2\varepsilon}  \\ k-\frac{C^{\nu}}{2\alpha}& >\frac{3\alpha C^{\nu}}{2}
\\ k^{\prime}-\frac{C_g^\nu \varepsilon}{2}& >\frac{3\alpha C^{\nu}}{2}. 
\end{aligned}\right.
\end{equation}
For these inequalities, noticing that $\frac{C_g^\nu \varepsilon}{2}+\frac{C_g^\nu}{2\varepsilon}$, $\frac{C^{\nu}}{2\alpha}+\frac{3\alpha C^{\nu}}{2} $ reach their minimum when $ \varepsilon = 1$ and $\alpha= \frac{\sqrt{3}}{3}$, respectively. To give a sufficient condition on $C_g^\nu,C^{\nu}$, we choose $\alpha= \frac{\sqrt{3}}{3}$, $ \varepsilon = 1$, and set $\gamma_1,\gamma_2 > 0$ to be the coefficients satisfying $C_g^\nu,C^{\nu} <  min\{\gamma_1 k, \gamma_2 k^{\prime}\}$.

If $\gamma_1 k \leq \gamma_2 k^{\prime}$, then, \eqref{contraction coeffcient 1} holds if the following system of inequalities hold.
\begin{equation}\label{coef 2}
\left\{
\begin{aligned}
&k^{\prime} > C_g^\nu\\
&k >\sqrt{3} C^{\nu} \\
&k-\frac{\sqrt{3}\gamma_1 k }{2} >\frac{\gamma_1 k }{2} \\
&k^{\prime}-\frac{\gamma_1 k  }{2} >\frac{\sqrt{3}\gamma_1 k }{2}.
\end{aligned}\right.
\end{equation}
From the third inequality, we obtain $\gamma_1 < \sqrt{3} - 1$. For the forth inequality in \eqref{coef 2}, it is enough to show 
\begin{equation*}
\begin{aligned}
\frac{\gamma_1 k}{\gamma_2}-\frac{\gamma_1 k  }{2}>\frac{\sqrt{3}\gamma_1 k }{2},
\end{aligned}
\end{equation*}
which means $\gamma_2 <\sqrt{3} - 1$. It is easily checked that we obtain the same result under the other condition $\gamma_1 k > \gamma_2 k^{\prime}$. Finally, compared with $C^{\nu} <\frac{\sqrt{3}}{3} k$, we obtain the sufficient condition
$C_g^\nu,C^{\nu}  < min \{(\sqrt{3}-1)k^{\prime}, \frac{\sqrt{3}}{3}k \}$ for which 
$\gamma(\epsilon, \delta, \alpha, \rho)<\theta(\epsilon, \delta, \alpha, \rho)$ when $\rho=\varepsilon=1$, $\alpha=\frac{\sqrt{3}}{3}$ and $\delta >0$ is small enough. 
\qed
\bigskip
 
\noindent (ii) \underline{Uniqueness of the solution}: 
Let $U^{\prime}=(X^{\prime},Y^{\prime}, Z^{\prime})$ be another solution to \eqref{mf-bfsde}. Set
\begin{equation}\label{Gamma-T}        
\begin{array}{lll}
\Gamma_T:=  \E\left[\int_0^T\left\{  (f(s,U^{\prime}_s,\nu^{\prime}_s)-f(s,U_s,\nu_s))\cdot(Y^{\prime}_s-Y_s) \right. \right. \\  \\ \qquad \qquad \qquad  \left. \left.
  +(h(s,U^{\prime}_s,\nu^{\prime}_s)-h(s,U_s,\nu_s))\cdot(X^{\prime}_s-X_s) \right. \right. \\  \\ \qquad \qquad \qquad \qquad  \left. \left.
+[\sigma(s,U^{\prime}_s,\nu^{\prime}_s)-\sigma(s,U_s,\nu_s), Z^{\prime}_s-Z_s] \right\}ds\right].
\end{array}
\end{equation}
Applying It\^o's formula to the product $(X^{\prime}_T-X_T)\cdot(Y^{\prime}_T-Y_T)$ and taking expectation, we obtain
\begin{equation}\label{uni-duality-1}
\E[(X^{\prime}_T-X_T)\cdot(Y^{\prime}_T-Y_T)] = \Gamma_T. 
\end{equation}
In view of (H1), \eqref{dist} and the Lipschitz continuity of $f, h, \sigma$ and $g$, and the Cauchy-Schwarz inequality we have
$$\begin{array}{lll}
\Gamma_T=\E[(X^{\prime}_T-X_T)\cdot(Y^{\prime}_T-Y_T)]\ge k^{\prime}\E[|X^{\prime}_T-X_T|^2]-C_g^\nu\E[|X^{\prime}_T-X_T|]d(\mu^{\prime}_T,\mu_T)\\ \\ \qquad\qquad\qquad\qquad\qquad \ge  k^{\prime} \E[|X^{\prime}_T-X_T|^2]-C_g^\nu\E[|X^{\prime}_T-X_T|]\E[|X^{\prime}_T-X_T|^2]^{\frac{1}{2}} 
\\ \\ \qquad\qquad\qquad\qquad\qquad \ge  (k^{\prime} -C_g^\nu)\E[|X^{\prime}_T-X_T|^2].
\end{array}
$$
Therefore,  
\begin{equation}\label{uni-duality-2}
\Gamma_T\ge (k^{\prime} -C_g^\nu)\E[|X^{\prime}_T-X_T|^2].
\end{equation}
On the other hand, we have
$$
\begin{array}{lll}
\Gamma_T\le \E\left[\int_0^T\left\{ \mathcal{A}(s,U_s,U^{\prime}_s,\nu_s)+(C^{\nu}|X^{\prime}_s-X_s|+C^{\nu}|Y^{\prime}_s-Y_s|+C^{\nu}\|Z^{\prime}_s-Z_s\|)d(\nu_s,\nu^{\prime}_s)\right\}ds\right].
\end{array}
$$
But 
$$
d(\nu_s,\nu^{\prime}_s)\le \sqrt{\E[|X^{\prime}_s-X_s|^2+|Y^{\prime}_s-Y_s|^2]}.
$$
Therefore, by the use of Young's inequality three times we obtain
\begin{align}\label{majorationgamma}
\Gamma_T\le &\E[\int_0^T\left\{ -k(|X^{\prime}_s - X_s|^2+|Y^{\prime}_s-Y_s|^2+\|Z^{\prime}_s-Z_s\|^2)+\frac{C^{\nu}}{2}(3\alpha+\frac{1}{\alpha})|X^{\prime}_s-X_s|^2 +
\right.\nn\\&\left.\qquad \qquad\frac{C^{\nu}}{2}(3\alpha+\frac{1}{\alpha})|Y^{\prime}_s-Y_s|^2 +
\frac{C^{\nu}}{2\alpha}\|Z^{\prime}_s-Z_s\|^2 \right\} ds ].
\end{align}
Now, combine \eqref{uni-duality-2} and \eqref{majorationgamma} to obtain
\begin{align*}
&(k^{\prime} -C_g^\nu)\E[|X^{\prime}_T-X_T|^2]+k\E\left[\int_0^T(|X^{\prime}_s-X_s|^2+|Y^{\prime}_s-Y_s|^2+\|Z^{\prime}_s-Z_s\|^2)ds\right]\nn \\&\qquad \qquad \leq \E[\int_0^T\left\{\frac{C^{\nu}}{2}(3\alpha+\frac{1}{\alpha})|X^{\prime}_s-X_s|^2 +
\frac{C^{\nu}}{2}(3\alpha+\frac{1}{\alpha})|Y^{\prime}_s-Y_s|^2 +
\frac{C^{\nu}}{2\alpha}\|Z^{\prime}_s-Z_s\|^2 \right\}ds]
\end{align*}
or
\begin{align*}
&0\leq(k^{\prime} -C_g^\nu)\E[|X^{\prime}_T-X_T|^2]\le \E[\int_0^T\left\{[\frac{C^{\nu}}{2}(3\alpha+\frac{1}{\alpha})-k]|X^{\prime}_s-X_s|^2 +\right.\\
&\qquad \qquad \qquad \qquad \left. 
[\frac{C^{\nu}}{2}(3\alpha+\frac{1}{\alpha})-k]|Y^{\prime}_s-Y_s|^2+
[\frac{C^{\nu}}{2\alpha}-k]\|Z^{\prime}_s-Z_s\|^2 \right\}ds].
\end{align*}
Noticing now that $C_g^\nu,C^{\nu}  \leq min \{(\sqrt{3}-1)k^{\prime}, \frac{\sqrt{3}}{3}k \} $, with $\alpha=\frac{\sqrt{3}}{3}$, all the coefficients of the right hand side of the above inequality are negative, which implies that, $\bP$-a.s. for all $0\le s\le T$, $X^{\prime}_s=X_s$, $Y^{\prime}_s=Y_s$, and $Z^{\prime}_s=Z_s$, $ds\otimes d\bP$-a.e. Thus the solution of \eqref{mf-bfsde} is unique.
\qed

\subsection{Existence and uniqueness results when $\sigma$ does not depend on the mean-field term}${}$
\medskip

Assuming $\sigma$ does not depend on $\mathbb{P}_{(X_t,Y_t)}$ i.e. the MF-BFSDE \eqref{mf-bfsde} becomes: $\forall t\le T$, 
\begin{equation}\label{mf-bfsde no P}
\left\{\begin{array}{lll}
X, Y \in \sdm \mbox{ and }Z\in \hdm ;\\\\
X_t=x+\int_0^t f(s,X_s,Y_s,Z_s,\mathbb{P}_{(X_s, Y_s)})ds+\int_0^t \sigma (s,X_s,Y_s,Z_s)dW_s,\quad t\le T,\\ \\
Y_t=g(X_T,\mathbb{P}_{X_T})-\int_t^Th(s,X_s,Y_s,Z_s,\mathbb{P}_{(X_s, Y_s)})ds-\int_t^T Z_sdW_s,\quad t\le T.
\end{array}
\right.
\end{equation}
In this framework, the condition (H1) can be relaxed to the following assumption. 
\begin{eqnarray*}
(H1^\prime)\left\{ \begin{array}{lll}
\text{(i) there exists $k >0$, s.t. for all } t\in[0,T], \nu\in\bM_2(\R^{m}\times\R^m), u,u^{\prime} \in \R^{m+m+m\times m}, \\
\quad \mathcal{A}(t,u,u^{\prime},\nu)\le -k(|x-x^{\prime}|^2+|y-y^{\prime}|^2),\quad \bP\as \\  \\
\text{(ii) there exists $k^{\prime} >0$, s.t. for all } \nu\in\bM_2(\R^{m}\times\R^m), x,x^{\prime}\in\R^m, \\
(g(x,\nu)-g(x^{\prime},\nu))\cdot(x-x^{\prime})\ge k^{\prime}|x-x^{\prime}|^2, \quad \bP\as  \\ 
\end{array}
\right.
\end{eqnarray*}

\medskip
\noindent
Following similar steps as in the proof of Theorem \ref{T123}, we have the following: 
 
\begin{theorem}[Existence and Uniqueness of a solution]\label{T123-2} Let Assumption $(H1^{\prime})$ hold. If the constants $C_g^\nu, C^{\nu}, k, k^{\prime}$ satisfy the inequalities 
\begin{equation}\label{coef condi-2}
C_g^\nu,C^{\nu}  < min \{2(\sqrt{2}-1)k^{\prime}, \frac{\sqrt{2}}{2}k \},
\end{equation}
then there exists a unique process $U=(X,Y,Z)$ which belongs to 
$\sdm\times \sdm\times \hdm$ and which solves the MF-BFSDE \eqref{mf-bfsde no P}.
\end{theorem}
\noindent \underline{Proof}: 
\noindent (i) \underline{Existence of a solution}:  Let $\delta>0$ and consider the sequence $U^{n}=(X^n,Y^n,Z^n)_{n\ge 0}$ of processes defined recursively as follows: 
$(X^0,Y^0,Z^0)=(0,0,0)$ and, for $n \ge 0, U^{n+1}$ satisfies,  for every $0\le t\leq T$, 
\begin{equation}\label{approximsanssima}\left\{
\begin{array}{lll}
U^{n+1}=(X^{n+1},Y^{n+1},Z^{n+1})\in \sdm\times \sdm\times \hdm \,\,; \\\\
X^{n+1}_t=x+\int_0^t \left\{f(s,U^{n+1}_s,\nu^{n}_s)-\delta (Y^{n+1}_s-Y^n_s)\right\}ds \\ \\ \qquad\qquad+\int_0^t \sigma (s,U^{n+1}_s)dW_s,\\ \\
Y^{n+1}_t=g(X^{n+1}_T,\mu^n_T)-\int_t^Th(s,U^{n+1}_s,\nu^n_s)ds-\int_t^T Z^{n+1}_sdW_s,
\end{array}\right.
\end{equation}
where $\nu^n_t:=\mathbb{P}_{(X^n_t, Y^n_t)}$ and $\mu^n_T:=\mathbb{P}_{X^n_T}$. By Theorem 1.2 in \cite{H98} (or \cite{Peng-wu}, pp.833), the system \eqref{approximsanssima} admits a unique solution. We will show that $(U^n)_{n\ge 0}$ is a Cauchy sequence in $\mathcal{M}^{2,m+m+m\times m}$ and $(X^n_T)_{n\ge 0}$ is a Cauchy sequence in $L^2(d\bP)$.
For $n\ge 1,\ t\in[0,T]$, recall the processes 
$\hat{X}^{n+1}$, $\hat{Y}^{n+1}$, $\hat{Z}^{n+1}$, 
$\widehat{\varphi}^{n+1}$ and $\overline{\varphi}^{n}$ defined respectively in \eqref{num1} and \eqref{num2}. 
\medskip

\noindent Applying It\^o's formula, we obtain
\begin{equation}\label{T2-duality-1}\begin{array}{lll}
\hat{X}^{n+1}_T\cdot\hat{Y}^{n+1}_T-\hat{X}^{n+1}_0\cdot\hat{Y}^{n+1}_0=\int_0^T \hat{Y}^{n+1}_s\cdot \{\hat{f}^{n+1}(s)-\delta(\hat{Y}^{n+1}_s-\hat{Y}^{n}_s)\}ds \\ \\ \qquad\qquad\qquad\qquad\qquad\qquad
+\int_0^T \hat{Y}^{n+1}_s\cdot \hat{\sigma}^{n+1}(s)dW_s \\ \\ \qquad\qquad\qquad\qquad\qquad\qquad 
+\int_0^T \hat{X}^{n+1}_s\cdot \hat{h}^{n+1}(s)ds+ \int_0^T \hat{X}^{n+1}_s\cdot \hat{Z}^{n+1}_sdW_s   \\ \\ \qquad\qquad\qquad\qquad \qquad\qquad
+\int_0^T [\hat{\sigma}^{n+1}(s), \hat{Z}^{n+1}_s]ds.
\end{array}
\end{equation}
Similarly as above, we  take expectation to obtain
\begin{equation}\label{T2-duality-2}\begin{array}{lll}
\E[\hat{X}^{n+1}_T\cdot(g(X^{n+1}_T,\mu_T^{n})- g(X^{n}_T,\mu_T^{n-1}))] +
\delta \E\left[\int_0^T  |\hat{Y}^{n+1}_s|^2 ds \right] \\ \\ \qquad
 -\E\left[\int_0^T\left( \hat{X}^{n+1}_s \cdot \hat{h}^{n+1}(s)+\hat{Y}^{n}_s\cdot\hat{f}^{n+1}(s)+[\hat{\sigma}^{n+1}(s),\hat{Z}^{n+1}_s]\right)ds\right] \\ \\
 \qquad= \delta \E\left[\int_0^T \hat{Y}^{n+1}_s\cdot \hat{Y}^{n}_s ds\right].
 \end{array}
\end{equation}
 \noindent Using the Lipschitz continuity of $g$,  Young's inequality, \eqref{dist} and (H$1^{\prime}$(ii)), we have, for any $\varepsilon>0$, 
\begin{equation}\label{n21}
\begin{array}{lll}
\E[\hat{X}^{n+1}_T \cdot (g(X^{n+1}_T,\mu_T^{n})- g(X^{n}_T,\mu_T^{n-1}))] =\E[\hat{X}^{n+1}_T\cdot (g(X^{n+1}_T,\mu_T^{n})- g(X^{n}_T,\mu_T^{n}))] \\ \\ 
\qquad\qquad\qquad+\E[\hat{X}^{n+1}_T\cdot (g(X^{n}_T,\mu_T^{n})- g(X^{n}_T,\mu_T^{n-1}))] \\ \\  
\qquad\qquad\qquad\ge k^{\prime}\E[|\hat{X}^{n+1}_T|^2]-C_g^\nu\E[|\hat{X}^{n+1}_T|]d(\mu_T^{n},\mu_T^{n-1}) \\ \\ 
 \qquad\qquad\qquad\ge  k^{\prime}\E[|\hat{X}^{n+1}_T|^2]-\frac{C_g^\nu\varepsilon}{2} \E[|\hat{X}^{n+1}_T|^2]-\frac{C_g^\nu}{2\varepsilon}d^2(\mu_T^{n},\mu_T^{n-1})  \\ \\ \qquad\qquad\qquad \ge (k^{\prime}-\frac{C_g^\nu\varepsilon}{2})\E[|\hat{X}^{n+1}_T|^2]-\frac{C_g^\nu}{2\varepsilon}\E[|\hat{X}^{n}_T|^2].
 \end{array}
\end{equation}
Again, by the Lipschitz continuity of $f, h,\sigma$,  Young's inequality, \eqref{dist} and (H$1^{\prime}$(i)), we also have, for every $0\le  t\le T$ and any $\alpha>0$,
 $$
 \begin{array}{lll}
&\hat{X}^{n+1}_t \cdot \hat{h}^{n+1}(t)+\hat{Y}^{n+1}_t\cdot \hat{f}^{n+1}(t)+[\hat{\sigma}^{n+1}(t),\hat{Z}^{n+1}_t]\\ \\&=\mathcal{A}(t,U_t^{n+1},U_t^{n},\nu^{n}_t)+\hat{X}^{n+1}_t\cdot\bar{h}^{n}(t)+\hat{Y}^{n+1}_t\cdot\bar{f}^{n}(t)\\  \\
&\le -k\left\{|\hat{X}^{n+1}_t|^2+|\hat{Y}^{n+1}_t|^2\right\}+|\hat{X}^{n+1}_t|| \bar{h}^{n}(t)|+|\hat{Y}^{n+1}_t||\bar{f}^{n}(t)|\\ \\ &
\le -k\left\{|\hat{X}^{n+1}_t|^2+|\hat{Y}^{n+1}_t|^2\right\}+C^{\nu} d(\nu_t^{n},\nu_t^{n-1})\left(|\hat{X}^{n+1}_t|
 +|\hat{Y}^{n+1}_t|\right)\\\ \\ &
\le -k\left\{|\hat{X}^{n+1}_t|^2+|\hat{Y}^{n+1}_t|^2\right\}+\frac{ C^{\nu}}{2\alpha}|\hat{X}^{n+1}_t |^2+\frac{ C^{\nu}}{2\alpha}|\hat{Y}^{n+1}_t |^2+\alpha C_\phi^\nu \cdot d^2(\nu_t^{n},\nu_t^{n-1}). 
\end{array}
$$
Since $d^2(\nu_t^{n},\nu_t^{n-1})\le \E[|\hat{X}^{n}_t|^2+|\hat{Y}^{n}_t|^2]$, then 
\begin{eqnarray}\label{n22}\begin{array}{lll}
\E\left[\int_0^T\left( \hat{X}^{n+1}_s \cdot \hat{h}^{n+1}(s)+\hat{Y}^{n+1}_s\cdot \hat{f}^{n+1}(s)+[\hat{\sigma}^{n+1}(s),\hat{Z}^{n+1}_s]\right)ds\right] \\ \\ \qquad \qquad\le 
\E\int_0^T \left( (\frac{ C^{\nu}}{2\alpha}-k)(|\hat{X}^{n+1}_t |^2+|\hat{Y}^{n+1}_t |^2)+\alpha C_\phi^\nu \E[|\hat{X}^{n}_t|^2+|\hat{Y}^{n}_t|^2]\right) ds.
\end{array}
\end{eqnarray}
Furthermore, we have, for any $\rho>0$, 
\begin{equation}\label{n23}
\begin{array}{lll}
\E\left[\int_0^T\hat{Y}^{n+1}_s\cdot \hat{Y}^{n}_s \ ds\right]\le 
\frac{1}{2}\E\left[\int_0^T\left(\rho |\hat{Y}^{n+1}_s|^2  +\frac{1}{\rho}|\hat{Y}^{n}_s|^2\right)ds\right].
\end{array}
\end{equation}
Applying now the three inequalities 
\eqref{n21}, \eqref{n22} and \eqref{n23} to \eqref{T1-duality-2}, yields
\begin{equation*}\begin{array}{lll}
&(k^{\prime}-\frac{C_g^\nu\varepsilon}{2})\E[|\hat{X}^{n+1}_T|^2]-\frac{C_g^\nu}{2\varepsilon}\E[|\hat{X}^{n}_T|^2] +
\delta \E\left[\int_0^T  |\hat{Y}^{n+1}_s|^2 ds \right] \\ \\ 
&- \E\left[\int_0^T\left\{ (\frac{ C^{\nu}}{2\alpha}-k)\cdot |\hat{X}^{n+1}_t |^2+(\frac{ C^{\nu}}{2\alpha}-k) \cdot|\hat{Y}^{n+1}_t |^2+\alpha C_\phi^\nu\E[|\hat{X}^{n}_t|^2+|\hat{Y}^{n}_t|^2] \right\}ds\right]
\\ \\ \qquad &\le \delta \E\left[\int_0^T\left(\frac{\rho}{2} |\hat{Y}^{n+1}_s|^2 + \frac{1}{2\rho}|\hat{Y}^{n}_s|^2\right)ds\right].
\end{array}
\end{equation*}
Rearranging terms, we obtain
\begin{equation*}\begin{array}{lll}
&(k^{\prime}-\frac{C_g^\nu\varepsilon}{2})\E[|\hat{X}^{n+1}_T|^2]+
\E\left[\int_0^T(k-\frac{C^{\nu}}{2\alpha})|\hat{X}^{n+1}_t|^2ds\right] +
\E\left[\int_0^T \left(\delta (1-\frac{\rho}{2} \right)+k-\frac{ C^{\nu}}{2\alpha})|\hat{Y}^{n+1}_s|^2ds \right] \\ \\&\le \frac{C_g^\nu}{2\varepsilon}\E[|\hat{X}^{n}_T|^2] +\E\left[\int_0^T\left(\alpha C^{\nu}|\hat {X}^{n}_s|^2+(\frac{\delta}{2\rho}+\alpha C^{\nu})|\hat{Y}^{n}_s|^2\right)ds\right].
\end{array}
\end{equation*}Taking $\rho=1$, we obtain
\begin{equation*}\begin{array}{lll}
&(k^{\prime}-\frac{C_g^\nu\varepsilon}{2})\E[|\hat{X}^{n+1}_T|^2]+
\E\left[\int_0^T(k-\frac{C^{\nu}}{2\alpha})|\hat{X}^{n+1}_t|^2ds\right] + \E\left[\int_0^T \left(\frac{1}{2}\delta +k-\frac{ C^{\nu}}{2\alpha})|\hat{Y}^{n+1}_s|^2\right)ds \right] \\  \\&\le \frac{C_g^\nu}{2\varepsilon}\E[|\hat{X}^{n}_T|^2] +\E\left[\int_0^T\left(\alpha C^{\nu}|\hat {X}^{n}_s|^2+(\frac{1}{2}\delta+\alpha C^{\nu})|\hat{Y}^{n}_s|^2\right)ds\right].
\end{array}
\end{equation*}
Let us set 
\begin{equation}\label{lb and thta}
\begin{aligned}
&\lambda(\epsilon, \delta, \alpha):=\min\{k^{\prime}-\frac{C_g^\nu \varepsilon}{2}, \frac{1}{2}\delta +k-\frac{C^{\nu}}{2\alpha}\}, \\&
\theta(\epsilon, \delta, \alpha):=\max\{\frac{C_g^\nu}{2\varepsilon}, \frac{1}{2}\delta+\alpha C^{\nu}\}.
\end{aligned}
\end{equation}
Then, it holds that
\begin{equation}\label{T2-contraction2}
\begin{array}{lll}
\E[|\hat{X}^{n+1}_T|^2]+\E\left[\int_0^T |\hat{X}^{n+1}_t|^2+|\hat{Y}^{n+1}_t|^2 ds\right]\le \frac{\theta}{\lambda}\left(\E[|\hat{X}^{n}_T|^2]+\E\left[\int_0^T |\hat{X}^{n+1}_t|^2+|\hat{Y}^{n+1}_t|^2ds\right] \right).
\end{array}
\end{equation}
Now, if there exist $\alpha$, $\varepsilon$, $\delta$ so that $\theta<\lambda$, the inequality 
\eqref{T2-contraction2} becomes a contraction. Thus,  $(X_T^n)_{n\ge 0}$ is a Cauchy sequence in $L^2(\Omega,\bP)$ and $(X^n)_{n\ge 0}$ and $(Y^n)_{n\ge 0}$  are Cauchy sequences in  $L^2([0,T]\times\Omega, dt\otimes d\bP)$. \\
To make the contraction meaningful, we assume $k^{\prime}-\frac{C_g^\nu \varepsilon}{2}$ and $k-\frac{C^{\nu}}{2\alpha}$  are positive. 
Next, similarly to Theorem \ref{T123}, since $\delta > 0$ can be chosen small enough,  we only need to solve the following system of inequalities (which stem from the limits of $\lambda$ and $\theta$ as $\delta \rightarrow 0$):
\begin{equation}\label{T2-contraction coeffcient}
\left\{\begin{aligned}
k^{\prime}-\frac{C_g^\nu \varepsilon}{2}& >\frac{C_g^\nu}{2\varepsilon} \\
k-\frac{C^{\nu}}{2\alpha}& >\frac{C_g^\nu}{2\varepsilon}  \\ 
k-\frac{C^{\nu}}{2\alpha}& >\alpha C^{\nu} \\
k^{\prime}-\frac{C_g^\nu \varepsilon}{2}& >\alpha C^{\nu} .
\end{aligned}\right.
\end{equation}
As in the proof of Theorem \ref{T123}, for those inequalities, we choose $\alpha= \frac{\sqrt{2}}{2}$, $ \varepsilon = 1$ and set $\gamma_3,\gamma_4 > 0$ to be the coefficients satisfying the $C_g^\nu,C^{\nu}<  min\{\gamma_3 k, \gamma_4 k^{\prime}\}$. 
Assuming $\gamma_3 k \leq \gamma_4 k^{\prime}$, \eqref{T2-contraction coeffcient} holds if the following system of inequalities holds:

\begin{equation}\label{T2-coef 2}
\left\{
\begin{aligned}
&k^{\prime} > C_g^\nu\\
&k >\sqrt{2} C^{\nu} \\
&k-\frac{\sqrt{2}\gamma_3 k }{2} > \frac{\gamma_3 k }{2} \\
&k^{\prime}-\frac{\gamma_3 k  }{2} > \frac{\sqrt{2}\gamma_3 k }{2}.
\end{aligned}\right.
\end{equation}
From the third inequality, we obtain $\gamma_3 < 2(\sqrt{2} - 1)$. For the forth inequality in \eqref{T2-coef 2} to be satisfied, it is enough to have
\begin{equation*}
\begin{aligned}
\frac{\gamma_3 k}{\gamma_4}-\frac{\gamma_3 k  }{2} >\frac{\sqrt{2}\gamma_3 k }{2}
\end{aligned}
\end{equation*}
which means $\gamma_4 <2(\sqrt{2} -1 )$ . The result under the other condition $\gamma_1 k > \gamma_2 k^{\prime}$ turns out to be the same, as it can be easily checked. Finally, compared with $C^{\nu} <\frac{\sqrt{2}}{2} k$ , we obtain a sufficient condition $C_g^\nu,C^{\nu}  < min \{2(\sqrt{2}-1)k^{\prime}, \frac{\sqrt{2}}{2}k \}$
satisfying $\gamma(\epsilon, \delta, \alpha, \rho)<\theta(\epsilon, \delta, \alpha, \rho)$ , when $\varepsilon=1$, $\alpha=\frac{\sqrt{2}}{2}$ and $\delta >0$ is small enough.

\medskip
We now show that $(Z^n)_{n\ge 0}$ is also a Cauchy sequence in $L^2([0,T]\times\Omega, dt\otimes d\bP)$ .\\ For $n,m\ge 0$,  we have 
 \begin{equation*}
\begin{aligned}
d(Y^n_t-Y^m_t)=(h(t,U^{n}_t,\nu^{n-1}_t)-h(t,U^{m}_t,\nu^{m-1}_t))dt+(Z^n_t-Z^m_t)dW_t.
\end{aligned}
\end{equation*}
By applying the It\^o formula and then taking expectation, we obtain
  \begin{equation*}
\begin{aligned}
\E[|Y^n_T-Y^m_T|^2-|Y^n_t-Y^m_t|^2]=\E[\int_t^T\left\{2|Y^n_s-Y^m_s||h(s,U^{n}_s,\nu^{n-1}_s)-h(s,U^{m}_s,\nu^{m-1}_s)|+\|Z^n_s-Z^m_s\|^2\right\}ds].
\end{aligned}
\end{equation*}
In view of the Lipschitz condition on the coefficients and Young's inequality, we have, for any $\beta>0$,
 \begin{equation*}
\begin{aligned}
\E\int_t^T\|Z^n_s&-Z^m_s\|^2ds=\E[|Y^n_T-Y^m_T|^2-|Y^n_t-Y^m_t|^2]+\E\int_t^T2|Y^n_s-Y^m_s||h(t,U^{n}_s,\nu^{n-1}_s)-h(t,U^{m}_s,\nu^{m-1}_s)|ds\\
&\leq \E[|Y^n_T-Y^m_T|^2]+\E[\int_t^T2C^u|Y^n_s-Y^m_s|\left\{|X^n_s-X^m_s|+|Y^n_s-Y^m_s|\right.\\
 &\qquad +\left.\|Z^n_s-Z^m_s\|\right\}+2C^{\nu} |Y^n_s-Y^m_s|d(\nu_s^{n-1}-\nu_s^{m-1}) ds]\\
&\leq  \E[|Y^n_T-Y^m_T|^2]+\E[\int_t^T\left\{2C^u|Y^n_s-Y^m_s|\left\{|X^n_s-X^m_s|+|Y^n_s-Y^m_s|\right\}+\frac{1}{2\beta}\|Z^n_s-Z^m_s\|^2\right.\\
 &\left.\qquad +2\beta(C^u)^2|Y^n_s-Y^m_s|^2+C^{\nu} \{(|Y^n_s-Y^m_s|)^2+\E[|X^{n-1}_s-X^{m-1}_s|^2+|Y^{n-1}_s-Y^{m-1}_s|^2]\}\right\} ds].\\
\end{aligned}
\end{equation*}
Let $\beta=1$ and $t=0$. We then have
 \begin{equation}\label{estimat Z}
 \begin{aligned}
\frac{1}{2}\E\int_0^T\|Z^n_s&-Z^m_s\|ds \leq  \E[|Y^n_T-Y^m_T|^2]+\E\left [\int_0^T\left\{2C_\phi^u|Y^n_s-Y^m_s|\left\{|X^n_s-X^m_s|+|Y^n_s-Y^m_s|\right\}\right.\right.\\
 &\left.+\left.2(C^u)^2|Y^n_s-Y^m_s|^2+C^\nu \{|Y^n_s-Y^m_s|^2+\E[|X^{n-1}_s-X^{m-1}_s|^2+|Y^{n-1}_s-Y^{m-1}_s|^2]\}\right\} ds\right ].
\end{aligned}
\end{equation}
Since  $(X_T^n)_{n\ge 0}$ is a Cauchy sequence in $L^2(\Omega,\bP)$ and $(X^n)_{n\ge 0}$ and $(Y^n)_{n\ge 0}$  are Cauchy sequences in  $L^2([0,T]\times\Omega, dt\otimes d\bP)$,  
$(Z^n)_{n\ge 0}$ is also a Cauchy sequence in $L^2([0,T]\times\Omega, dt\otimes d\bP)$ . Therefore, by standard calculations, we also have 
$$\E[\sup_{s\le T}(|X^n_s-X^m_s|^2+|Y^n_s-Y^m_s|^2)]\rightarrow 0\mbox{ as }n,m\rightarrow\infty.$$
Consequently, there exist $\mathbb{F}$-adapted continuous processes $X$ and $Y$ and an $\mathbb{F}$-progressively measurable process $Z$ such that
$$\E[\sup_{s\le T}(|X^n_s-X_s|^2+|Y^n_s-Y_s|^2)+\int_0^T\|Z^n_s-Z_s\|^2ds]\rightarrow 0\,\,\, \mbox{ as }\,\, n\rightarrow\infty.$$
Moreover, 
$$\E[\sup_{s\le T}(|X_s|^2+|Y_s|^2)+\int_0^T\|Z_s\|^2ds]<\infty .$$
By taking the limit with respect to $n$ in equation \eqref{approximsanssima} we obtain that $(X,Y,Z)$ is a solution of MF-BFSDE \eqref{mf-bfsde no P}. 
\medskip

\noindent (ii) \underline{Uniqueness of the solution}: 
Let $U^{\prime}=(X^{\prime},Y^{\prime},Z^{\prime})$ be another solution to \eqref{mf-bfsde no P} and set 
\begin{equation}\label{Gamma-T1}
\begin{array}{lll}
\Gamma_T:=  \E\left[\int_0^T\left\{  (f(s,U^{\prime}_s,\nu^{\prime}_s)-f(s,U_s,\nu_s))\cdot(Y^{\prime}_s-Y_s) \right. \right. \\  \\ \qquad \qquad \qquad  \left. \left.
  +(h(s,U^{\prime}_s,\nu^{\prime}_s)-h(s,U_s,\nu_s))\cdot(X^{\prime}_s-X_s) \right. \right. \\  \\ \qquad \qquad \qquad \qquad  \left. \left.
+[\sigma(s,U^{\prime}_s)-\sigma(s,U_s), Z^{\prime}_s-Z_s] \right\}ds\right].
\end{array}
\end{equation}
Applying It\^o's formula to the product $(X^{\prime}_T-X_T)\cdot(Y^{\prime}_T-Y_T)$ and taking expectation, we obtain
\begin{equation}\label{T2-uni-duality-1}
\E[(X^{\prime}_T-X_T)\cdot(Y^{\prime}_T-Y_T)] = \Gamma_T. 
\end{equation}
In view of $(H1^{\prime})$, \eqref{dist}, the Lipschitz continuity of $f, h, \sigma$ and $g$, and the Cauchy-Schwarz inequality, we have
$$\begin{array}{lll}
\Gamma_T=\E[(X^{\prime}_T-X_T)\cdot(Y^{\prime}_T-Y_T)]\ge k^{\prime}\E[|X^{\prime}_T-X_T|^2]-C_g^\nu\E[|X^{\prime}_T-X_T|]d(\mu^{\prime}_T,\mu_T)\\ \\ \qquad\qquad\qquad\qquad\qquad \ge  k^{\prime} \E[|X^{\prime}_T-X_T|^2]-C_g^\nu\E[|X^{\prime}_T-X_T|]\E[|X^{\prime}_T-X_T|^2]^{\frac{1}{2}} 
\\ \\ \qquad\qquad\qquad\qquad\qquad \ge  (k^{\prime} -C_g^\nu)\E[|X^{\prime}_T-X_T|^2].
\end{array}
$$
Therefore,  
\begin{equation}\label{T2-uni-duality-2}
\Gamma_T\ge (k^{\prime} -C_g^\nu)\E[|X^{\prime}_T-X_T|^2].
\end{equation}
On the other hand, we have
$$
\begin{array}{lll}
\Gamma_T\le \E\left[\int_0^T\left\{ \mathcal{A}(s,U_s,U^{\prime}_s,\nu_s)+(C^{\nu}|X^{\prime}_s-X_s|+C^{\nu}|Y^{\prime}_s-Y_s|)d(\nu_s,\nu^{\prime}_s)\right\}ds\right].
\end{array}
$$
But, 
$$
d(\nu_s,\nu^{\prime}_s)\le \sqrt{\E[|X^{\prime}_s-X_s|^2+|Y^{\prime}_s-Y_s|^2]}.
$$
Therefore, by the use of Young's inequality three times we obtain
\begin{align}\label{T2-majorationgamma2}
\Gamma_T\le \E\int_0^T\left\{ -k(|X^{\prime}_s - X_s|^2+|Y^{\prime}_s-Y_s|^2)+\frac{C^{\nu}}{2}(2\alpha+\frac{1}{\alpha})(|X^{\prime}_s-X_s|^2 +|Y^{\prime}_s-Y_s|^2)
\right\} ds. 
\end{align}
Now, combining \eqref{T2-uni-duality-2} and \eqref{T2-majorationgamma2} we obtain
\begin{equation}\label{uni-ineq2}
\begin{aligned}
&0\leq(C_g^\nu-k^{\prime} )\E[|X^{\prime}_T-X_T|^2]\le \E[\int_0^T\left\{[\frac{C^{\nu}}{2}(2\alpha+\frac{1}{\alpha})-k]|X^{\prime}_s-X_s|^2 +\right.\\
&\qquad\qquad\qquad\qquad\qquad\qquad\qquad\qquad\left. 
[\frac{C^{\nu}}{2}(2\alpha+\frac{1}{\alpha})-k]|Y^{\prime}_s-Y_s|^2 \right\}ds.
\end{aligned}
\end{equation}
Since $C_g^\nu,C^{\nu}  < min \{2(\sqrt{2}-1)k^{\prime}, \frac{\sqrt{2}}{2}k \} $, with $\alpha=\frac{\sqrt{2}}{2}$, all the coefficients in the inequality \eqref{uni-ineq2} are negative. It follows that $\bP$-a.s. for any $s\in [0,T]$, $X^{\prime}_s=X_s$, $Y^{\prime}_s=Y_s$, and using the estimation \eqref{estimat Z}, we obtain $Z^{\prime}_s=Z_s$, $ds\otimes d\bP$, which is the uniqueness of the solution of \eqref{mf-bfsde no P}.
\qed 

\begin{remark} $ \,$
\begin{itemize}
\item[(i)] Conditions \eqref{coef condi} and \eqref{coef condi-2} in Theorems \ref{T123} and \ref{T123-2} are only sufficient conditions. Whether or not they are necessary does not seem an easy task. However, as it is well known, the existence of a solution for the MF-BFSDEs \eqref{mf-bfsde} and \eqref{mf-bfsde no P} depends on several parameters including the length $T$ of the time horizon and the initial value $x$ of the forward SDE (see Example \eqref{FBSDE of example} below). 

\item[(ii)] Conditions \eqref{coef condi} and \eqref{coef condi-2}  can be improved if we consider $C^\nu_f, C^\nu_\sigma, C^\nu_h$ instead of $C^{\nu}=\max\{C^\nu_f, C^\nu_\sigma, C^\nu_h \}$. 
\end{itemize}
\end{remark}

\section{The nonzero-sum mean filed game: the open-loop framework}
In this section $W=(W_t)_{t\leq T}$ is a one-dimension Brownian motion. For $i=1,\dots, m$, let ${\U}^i:=\M^{2,m_i}$, be the set of open-loop admissible controls  for the player $i$. The set ${\U}:=\Pi_{i=1,m} {\U}^i$, is called of open-loop admissble strategies for the players. In the sequel, a stochastic process $\rho=(\rho_t(\omega))_{t\leq T}$ with values in $\R^{\ell_1\times \ell_2}$ is called bounded if \begin{equation}\lb{defnorm}\|\rho\|:=\sup_{(t,\omega)\in [0,T]\times \Omega}\|\rho_t(\omega)\|<\infty.
\end{equation}
Next, for $u=(u^i)_{1\le i\le m}\in \U$, let $X^u:=(X^u_t)_{0\le t\leq T}$ be the $\R^n$-valued process solution of the following standard SDE of mean-field or McKean-Vlasov type.
\begin{equation}
X^u_t=x+\int_0^t\{A_sX^u_s+\sum_{k=1,m}C^k_su^k_s+D_s \E[X^u_s]+\beta_s\}ds+\int_0^t\{\sigma_s X^u_s+\alpha_s\}dW_s,
\end{equation}
where,
\begin{enumerate}
\item[(i)] $A=(A_t)_{0\le t\leq T}$, $D=(D_t)_{0\le t\leq T}$, $\beta=(\beta_t)_{0\le t\leq T}$, $\alpha=(\alpha_t)_{0\le t\leq T}$ and $C^k=(C^k_t)_{0\le t\leq T}$ are bounded and adapted stochastic processes with values respectively in 
$\R^{n\times n}$, $\R^{n\times n}$, $\R^{n}$, $\R^{n}$ and $\R^{n\times m_k}$, $k=1,\ldots,m$.  

\item[(ii)] $\sigma=(\sigma_t)_{0\le t\leq T}$ is an adapted process with values in $\R^{n\times n}$. 
\end{enumerate}
\medskip

Next, to $u=(u^i)_{1\le i\le m}\in \U$, we associate $m$ payoffs $J_i(u)$, $i=1,\ldots,m,$ of the form
\begin{eqnarray*}\begin{array}{lll}
J_i(u):=\frac{1}{2}\{\E[(X^u_T)^\top Q^iX^u_T]+\E[(X^u_T)^\top]R^i \E[X^u_T] \\\\ \qquad\qquad\qquad\qquad\qquad +\E[\int_0^T\{(X^u_s)^{\top} M^i_sX^u_s+ u_s^{\top}N_s^iu_s+\E[X^u_s]^{\top} \Gamma^i_s\E[X^u_s]\}ds]\},
\end{array}
\end{eqnarray*}
where, for any $i=1,\ldots,m$,
\begin{itemize}

\item[(a)] $M^i=(M^i_t)_{0\le t\leq T}$ are bounded adapted symmetric non-negative matrices with values in $\R^{n\times n}$,  

\item[(b)] $N^i=(N^i_t)_{0\le t\leq T}$ are bounded adapted symmetric positive matrices with values $\R^{m_i\times m_i}$. Moreover, their inverses  
$(N^i)^{-1}:=((N^i_t)^{-1})_{0\le t\leq T}$ are also bounded,

\item[(c)] $(\Gamma^i_t)_{0\le t\le T}$ are bounded deterministic symmetric non-negative  matrices with values in $\R^{n\times n}$, 

\item[(d)] $Q^i$  is a random bounded symmetric non-negative matrix  $\cF_T$-measurable and $R^i$ is a constant symmetric non-negative matrix, with values in $\R^{n\times n}$.
\end{itemize} 
\bs

\noindent For $i=1,\ldots,m$, $J_i(u)$ is the cost associated with the player $i$ when the collective strategy $u=(u^i)_{1\le i\le m}$ is implemented. The problem we address in this section is to find a Nash equilibrium point (NEP) for the game, i.e., a collective control $u^*=(u_1^*,\ldots,u^*_n)$ for the players such that 
for any $i=1,\ldots,m$, 
\begin{equation}
\lb{ineqjeu}J_i(u_1^*,\ldots,u_m^*)\leq J_i(u_1^*,\ldots,u_{i-1}^*,u_i, u_{i+1}^*,\ldots,u_m^*), \,\,\mbox{for all}\,\, u_i\in \U^i.
\end{equation}
The meaning of the previous inequalities is that if the player $i$ makes the decision to deviate unilaterally from $u_i^*$, then she is penalized since her cost is at least larger than the cost of using $u_i^*$. If $m=2$ and $J_1+J_2=0$, the game is called of zero-sum type and a NEP $(u^{*}_1, u^{*}_2)$ satisfies
$$
J_1(u^{*}_1,v_2)\leq  J_1(u^{*}_1, u^{*}_2)\leq J_1(v_1,u^{*}_2)
$$for all $v_1\in {\U}^1$ and 
$v_2\in {\U}^2$. 

For the sake of simplicity, in this section we only consider the case where the Brownian motion is one-dimensional. Extension to the multi-dimensional case is straightforward.

For $i=1,\ldots,m$, let us denote by $H_i$  the Hamiltonian associated with the $i$-th player which is defined by
\begin{equation*}\begin{array}{lll}
H_i(t,\omega,x,u_1,\ldots,u_m,\zeta,p_i,q_i):=p_i^{\top}(A_t(\omega)x+\sum_{k=1}^m C^k_t(\omega)u^k+D_t\zeta +\beta_t)\\ \\ \qquad\qquad\qquad\qquad  +\frac{1}{2}(x^{\top}M^i_t(\omega)x+u^{\top}_i N^i_t(\omega)u_i+\zeta^{\top} \G^i_t \zeta)+ (\sigma_t^\top x+\alpha_t)q_t^i ,
\end{array}
\end{equation*}
where $u^i\in \R^{m_i}$, $z^i\in \R^n$ and $\zeta \in \R^n$ ($\zeta$ is the variable which stands for the expectation). 

For $i=1,\ldots,m$, let $\tilde u^i$ be the functions defined by
\begin{equation}\label{u-opti}
\tilde u^i(t,\omega,p^i):=-(N^i_t)^{-1}(C^i_t)^\top p_i, \quad 0\le t\leq T.
\end{equation}
The measurable functions $\tilde u^i,\,\, i=1,\ldots,m$, satisfy for all $i=1,\ldots,m$ and all $u^i\in \R^{m_i}$
$$
\begin{array}{lll}
H_i(t,\omega,x,(\tilde u^j(t,\omega,p^j))_{1\le j\le m},\zeta,p_i,q_i)\\\\ \quad \leq H_i(t,\omega,x,\tilde u^1(t,\omega,p^1),\ldots,\tilde u^{i-1}(t,\omega,p^{i-1}),u^i,\tilde u^{i+1}(t,\omega,p^{i+1}),\ldots,\tilde u^{m}(t,\omega,z^{m}),\zeta,p_i,q_i).
\end{array}
$$

The following proposition is a first step toward the proof of existence of a NEP for the game.
\begin{proposition}\label{prop33}
Let the $\cP-$measurable processes $(X,(p^1,q^1),\ldots,(p^m,q^m))$ be such that $X,p^i, i=1,\dots, m$ belong to $\mathcal{S}^{2,n}$ and $q^i, i=1,\dots, m$ belong to $\mathcal{H}^{2,n}$. Then, they solve the following MF-BFSDE, for all $0\le t\leq T$,
\begin{eqnarray}\label{eqbfsde}\left\{\begin{array}{l}
X_t=x+\int_0^t\{A_sX_s+\sum_{k=1}^mC^k_s\tilde u^k(s,p^k_s)+D_s\E[X_s]+\beta_s\}ds\\ \\
\qquad\qquad +\int_0^t\{\sigma_s X^u_s+\alpha_s\}dW_s;\\ \\

p^i_t=(Q^iX_T+R^i\E[X_T])+\int_t^T\{A^\top_sp^i_s+M^i_sX_s+\E[D^\top_sp^i_s]\\ \\ \qquad\qquad\qquad+\G^i_s  \E[X_s]+\sigma_s^\top q^i_s\}ds-\int_t^Tq^i_sdW_s,\,\, i=1,\ldots, m.
\end{array}
\right.
\end{eqnarray}
if and only if the admissible collective control $\tilde u:=(\tilde u^j)_{1\le j\le m}=((\tilde u^{j}(t,\omega,p^{j}_t))_{0\le t\leq T})_{1\le j\le m}$ ($\tilde u^{j}$ is given by \eqref{u-opti}) is a Nash equilibrium point for the mean-field nonzero-sum linear quadratic stochastic differential game.
\end{proposition}
In the BFSDE \eqref{eqbfsde}, $X$ is the optimal trajectory and  $(p_i,q_i)_{1\le i\le m}$ are the associated adjoint processes (\cite{AD, cadelinaskaratzas95, bensoussan81}).

\begin{proof}
\noindent (i) \underline{The condition is sufficient}. The fact that $\tilde u$ is an open-loop strategy for the players is an immediate consequence of the boundedness of $C^i_t$, $(N^i_t)^{-1}$ and the fact that $(p^i_t)_{0\le t\le T}$ belongs to $\M^{2,n}$ for any $i=1,\ldots,m$. Next, we will show the inequality \eqref{ineqjeu} for $i=1$. The other cases can be treated in the same manner. Consider $u_1=(u_1(s))_{0\le s\leq T}\in \U^{1}$, $\hat u=(u_1,\tilde u_2,\ldots,\tilde u_m)$. We should show that $J_1(\tilde u)\leq J_1(\hat u)$. 

\noindent Indeed,
\begin{align*}&J_1(\hat u)-J_1(\tilde u)=J_1(u_1,\tilde u_2,\ldots,\tilde u_m)-
J_1(\tilde u)=\\&
\frac{1}{2}\{\E[(X^{\hat u}_T)^{\top} Q^1X^{\hat u}_T]+\E[(X^{\hat u}_T)^{\top}]R^1 \E[X^{\hat u}_T]]-\E[(X_T)^{\top} Q^1 X_T+\E[(X_T)^{\top}]R^1 \E[X_T]]\}\\&
\qd +\frac{1}{2}\E[\int_0^T\{(X^{\hat u}_s)^{\top}M^1_sX^{\hat u}_s+u_1(s)^{\top}N_s^1u_1(s)-X_s^{\top}M^1_sX_s-\tilde u_1(s)^{\top}N_s^1.u_1(s)\\&\qd+\E[X^{\hat u}_s]^{\top}\Gamma^1_s\E[X^{\hat u}_s]-\E[X_s]^{\top} \Gamma^1_s\E[X_s]\}ds].
\end{align*}
But, for any  symmetric non-negative matrix $\Sigma\,$ (i.e. $v^{\top} \Sigma v\ge 0,\,\,\fr v\in \R^k$), we have   
$$
\begin{array}{ll}\theta_1^{\top}\Sigma\theta_1-\theta_2^{\top}\Sigma\theta_2&=
(\theta_1-\theta_2)^{\top}\Sigma(\theta_1-\theta_2)+2(\theta_1-\theta_2)^{\top}\Sigma
\theta_2\geq 2(\theta_1-\theta_2)^{\top}\Sigma
\theta_2.\end{array}
$$
Therefore, 
\begin{align}\lb{eqju}\nn J_1(\hat u)-J_1(\tilde u)\ge & \E\{(X^{\hat u}_T-X_T)^{\top} Q^1X_T\}+\E[(X^{\hat u}_T-X_T)^{\top}]R^1 \E[X_T]\\&\nn
\qd +\E\{\int_0^T\{(X^{\hat u}_s-X_s)^{\top}M^1_s.X_s+(u_1(s)-\tilde u_1(s))^{\top}N_s^1\tilde u_1(s)\\&\qd+(\E[X^{\hat u}_s]-\E[X_s])^{\top} \Gamma^1_s\E[X_s]\}ds\}.
\end{align}
since the matrices $Q^1$, $R^1$, $M^1_t$, $N^1_t$ and $\G^1_t$ are symmetric non-negative. 

We will show that the right-hand side of \eqref{eqju} is zero. Indeed, since $p^1_T=Q^1X_T+R^1\E[X_T]$ and $(p^1,q^1)$ is a solution of a backward SDE of mean-field type, then by It\^o's formula we have
\begin{eqnarray}\lb{dualite}\begin{array}{lll}
(X^{\hat u}_T-X_T)^{\top} p^1_T=\int_0^T\{-(X^{\hat u}_s-X_s)^{\top} \{A^\top_sp^1_s+M^1_s X_s+\E[D_s^{\top} p^1_s]+\G^1_s  \E[X_s]+\sigma_s^{\top} q^1_s\}\\ \\ \qquad\qquad +
(X^{\hat u}_s-X_s)^\top A^{\top}_sp^1_s+(u^1(s)-\tilde u^1(s))^{\top} (C^1_s)^{\top} p_s^1+\E[X^{\hat u}_s-X_s]^{\top} D_s^{\top} p_s^1\}ds\\ \\ \qquad\qquad
+\int_0^t
(X^{\hat u}_s-X_s)^{\top}\sigma_s^{\top} q^1_sds +
\int_0^t(X^{\hat u}_s-X_s)^{\top}\sigma_s^{\top} p^1_s dW_s+\\  \\ \qquad\qquad
\int_0^t(X^{\hat u}_s-X_s)^{\top} q^1_s dW_s,
\end{array}
\end{eqnarray}
since, for any $0\le t\leq T$, 
\begin{align}\nn
\begin{array}{l}X^{\hat u}_t-X_t=\int_0^t\{A_s(X^{\hat u}_s-X_s)+C^1_s(u^1(s)-\tilde u^1(s))+D_s\E[X^{\hat u}_s-X_s]\}ds+\int_0^t\sigma_s (X^{\hat u}_s-X_s)dW_s.\end{array}
\end{align}
Simplifying terms in \eqref{dualite} and taking expectation, noting that the stochastic integrals are martingales, we obtain 
\begin{eqnarray}\lb{dualite2}\begin{array}{ll}
\E[(X^{\hat u}_T-X_T)^\top p^1_T]=
\E\{(X^{\hat u}_T-X_T)^\top Q^1X_T]+\E[(X^{\hat u}_T-X_T)^\top]R^1 \E[X_T]\}\\
\qquad\qquad =\E[\int_0^T\{-(X^{\hat u}_s-X_s)^\top \{M^1_s X_s+\G^1_s  \E[X_s]\}+
(u^1(s)-\tilde u^1(s))^\top(C^1_s)^\top p^1_s\}ds].
\end{array}
\end{eqnarray}
Finally, insert the right-hand side of \eqref{dualite2} in \eqref{eqju} and take into account that $$
(C^1_s)^\top p^1_s+N_s^1\tilde u^1(s)=0
$$
(see the definition of $\tilde u^1$ given by \eqref{u-opti} ) to obtain that 
$$
J_1(u_1,\tilde u_2,\ldots,\tilde u_m)-J_1(\tilde u)\geq 0.
$$ 
\noindent (ii) \underline{The condition is necessary}.
Suppose the game has a Nash equilibrium point $\tilde u:=(\tilde u^j)_{1\le j\le m}=((\tilde u^{j}(t,\omega,p^{j}_t))_{0\le t\leq T})_{1\le j\le m}$ and denote by $\tilde X $ its associated optimal trajectory. Then obviously $\tilde X$ belongs to $\mathcal{S}^{2,n}$. Next, for $i=1,\ldots,m$, let $(p^i, q^i)$ be the solution of the following backward SDE:
\begin{equation}\label{proof of sf of mfbsde and ad eq}
\left\{\begin{aligned}
&p^i \in \mathcal{S}^{2,n} \mbox{ and }q^i\in \mathcal{H}^{2,n} ; \\
&p^i_t=(Q^i\tilde X_T+R^i\E[\tilde X_T])+\int_t^T\{A^\top_sp^i_s+M^i_s\tilde X_s+\E[D^\top_sp^i_s]\\ &\qquad\qquad +\G^i_s  \E[\tilde X_s]+\sigma_s^\top q^i_s\}ds-\int_t^Tq^i_sdW_s,\quad t\le T.
\end{aligned}\right.
\end{equation}The solution of \eqref{proof of sf of mfbsde and ad eq} exists, by the results in \cite{buckdahnlipeng}. Next, by the maximum principle (see \cite{AD}, Theorem 3.1),  for any $u:=( u^j)_{1\le j\le m} \in \mathcal{U}$, we have
\begin{equation}
\begin{aligned}
&\frac{d}{du^i} H_i(t,\omega,\tilde X_t,\tilde u_1,\ldots,\tilde u_{i-1},u_i,\tilde u_{i+1},\ldots,u_m,\E[\tilde X_t],p_i,q_i)(u_i - \tilde u_i) \geq 0\ \ \ \mathbb{P}\mbox{-a.s.} 
\end{aligned}
\end{equation}
for all \  $t \in [0, T]$, $i=1,\ldots, m$. That is, for all $i=1,\ldots, m$,
\begin{equation}
\begin{aligned}
((C_t^i)^\top p_t^i+ N_t^i \tilde u_t^i)(u_t^i-\tilde u_t^i) \geq 0.
\end{aligned}
\end{equation}
Since $u_t^i \in \R^{m_i}$ is arbitrary, we obtain 
$$\tilde u^i(t,\omega,p^i):=-(N^i_t)^{-1}(C^i_t)^\top p_i, \quad 0\le t\leq T. $$
Inserting that value of $\tilde{u}^i$ into \eqref{proof of sf of mfbsde and ad eq}, we have that $(\tilde X,(p^1,q^1),\ldots,(p^m,q^m))$ satisfies MF-BFSDE \eqref{eqbfsde}.
\end{proof}

Next, we are going to provide conditions on the data of the differential game in such a way that a NEP exists. So let us consider the following assumptions:

\begin{equation*}
(H2)
\left\{ \begin{array}{lll}
\text{(i) For any $i=1,\ldots,m$, the matrices $C^i$ and $N^i$ are time independent. We set } \\  \qquad \text{$K^i:=C^i(N^i)^{-1}(C^i)^\top$.}\\\\
\text{(ii) There exist constants $\eta_1>0$, $\eta_2>0$ such that for any $x\in \R^n$, }\\ \text{$\qquad x^\top (\sum_{i=1}^mK^iQ^i)x\ge \eta_1|x|^2$ and $x^\top (\sum_{i=1}^m K^iM_t^i)x\ge \eta_2|x|^2$,} \\ \text{for any $0\le t\leq T$, $\bP$-a.s.  
}\\\\

\text{(iii) For any $i=1,\ldots,m$, $K^iA_t^\top=A_t^\top K^i$, $K^iD^\top_t=
D^\top_tK^i$ and $K^i\sigma^\top_t=\sigma_t^\top K^i$},
\\ \text{for any $0\le t\leq T$, $\bP$-a.s. }
\end{array}
\right.
\end{equation*}

Note that in the case when  $n=1$, those assumptions are rather easy to check. 
\ms

Let us now consider the following MF-BFSDE.

\begin{eqnarray}\label{eqbfsde2}\left\{\begin{array}{l}
X, \tilde Y \in \mathcal{S}^{2,n} \mbox{ and }Z\in \mathcal{S}^{2,n}\,;\\\\
X_t=x+\int_0^t\{A_sX_s-\tilde Y_s+D_s\E[X_s]+\beta_s\}ds+\int_0^t\{\sigma_s X^u_s+\alpha_s\}dW_s\,,\\
\\
\tilde Y_t=(\sum_{k=1}^m K^iQ^i)X_T+
(\sum_{k=1}^m K^iR^i)\E[X_T]-\\\\\qquad \int_t^T\{-A^\top_s\tilde Y_s-(\sum_{k=1,m}K^iM^i_s)X_s-\E[D_s^\top\tilde Y_s]-\sigma_s^\top \tilde Z_s\}ds-\int_t^T\tilde Z_sdW_s.
\end{array}\right.
\end{eqnarray}

Note that  if $(X,(p^1,q^1),\ldots,(p^m,q^m))$ is a solution of \eqref{eqbfsde} then, under (H2), the process $(X,Y=\sum_{i=1}^m K^ip^i,Z=\sum_{i=1}^mK^iq^i)$ is a solution of the BFSDE \eqref{eqbfsde2}. This is exactly the origin of \eqref{eqbfsde2}.
\bs 

The functions $f$, $g$, $h$ and $\s$, introduced in Section 1, and related to the BFSDEs \eqref{eqbfsde2} are
\begin{itemize}

\item[(a)] $f(t,x,y,z,\nu)=A_tx-y+D_t\int_{\R^{n+n}}x\nu(dx,dy)+\beta_t$;
 
\item[(b)] $\s(t,x,y,z,\nu)=\s_t x+\alpha_t$ ; 

\item[(c)] $g(x,\mu)=(\sum_{k=1}^mK^iQ^i)x+
(\sum_{k=1}^mK^iR^i)\int_{\R^n}x\mu(dx).$

\item[(d)] For any $t,x,y,z,\nu$, if $(\xi_1,\xi_2)$ is a random vector on $(\Omega,\F,\bP)$ whose law is $\nu$, then $$h(t,x,y,z,\nu):=-A_t^\top y-(\sum_{k=1}^mK^iM^i_s)x-\E[D_t\xi_2]-\s^\top z.$$
\end{itemize}
To proceed, let us show that $f$ is uniformly Lipschitz w.r.t. $\nu$. Let $\xi=(\xi_1, \xi_2)$ be a random vector whose distribution is $\nu \in \bM_2(\R^{n+n})$. We have, $$
f(t,x,y,z,\nu)=A_tx-y+D_t\E[\xi_1]+\beta_t.$$ Next, let $\nu'\in \bM_2(\R^{n+n})$ be given and let $\xi'=(\xi'_1, \xi'_2)$ be a pair of random variables defined on the same probability as $(\xi_1, \xi_2)$ whose law is $\nu'$.  Therefore, 
\begin{align}
\nn |f(t,x,y,z,\nu)-f(t,x,y,z,\nu')|&\le \|D\||\E[\xi_1-\xi'_1]|\\&\nn\leq \|D\|\sqrt{\E[|\xi_1-\xi'_1|^2]}\\\nn&\leq \|D\|\sqrt{\E[|\xi-\xi'|^2]}.
\end{align}
Since $\xi$ and $\xi'$ are arbitrary, it holds that  
\begin{align}\lb{flipnu}
|f(t,x,y,z,\nu)-f(t,x,y,z,\nu')|\le \|D\|\inf_{\xi,\xi'}\sqrt{\E[|\xi-\xi'|^2]}=\|D\|d(\nu,\nu').
\end{align}
Finally, linearity implies that $f$ satisfies \eqref{lipschitz}. Similar estimates can be used for $h$ and $g$ to show that they satisfy \eqref{lipschitz} and  \eqref{lipschitzg}, respectively. 
\ms

\nd The operator $\mathcal{A}$ defined in \eqref{operateur} reads
\begin{equation}\lb{operateur2}
\begin{array}{lll}
\mathcal{A}(t,u,u^{\prime},\nu)=-|y-y'|^2-(\sum_{k=1}^mK^iM^i_s)|x-x'|^2. 
\end{array}
\end{equation}
Therefore, under (H2), $\mathcal{A}$ and $g$ satisfy Assumption (H$1^{\prime}$) with $k= min\{1, \eta_2\}$, $k^{\prime}= \eta_1$ ($\eta_1$ and $\eta_2$ are defined in (H2)).
\ms

\noindent Next, the Lipschitz constants of  $f$, $g$, $h$ and $\s$ w.r.t. $x,y$ and $z$ are (see \eqref{defnorm}) 
\begin{align}\lb{diversconstantes}
&C_f^x=\|A\|, \,\, C_f^y=1, \,\, C_f^z=0, \,\, C_h^y=\|A\|, \,\, C_h^x=\|(\sum_{k=1}^mK^iM^i)\|, \,\, C_h^z=\|\s\|,\\& C_\s^x=\|\s\|, \,\,C_\s^y=C_\s^z=0\text{ and }C_g^x=\|(\sum_{k=1}^mK^iQ^i)\|.
\end{align}
On the other hand, as for $f$ in \eqref{flipnu}, 
\begin{align}\lb{diversconstantes2}
C_f^\nu=\|D\|, C_h^\nu=\|D\|, C_\s^\nu=0 \text{ and }
C_g^\nu=\|(\sum_{k=1}^mK^iR^i)\|.
\end{align}
\noindent We have the following  
\begin{proposition}\label{Existence and uniqueness solution of game} Assume that (H2) holds and
\begin{equation}
\begin{aligned}
&(i)\qquad\|(\sum_{k=1}^mK^iR^i)\| < min\{2(\sqrt{2}-1)\eta_1,\frac{\sqrt{2}}{2},\frac{\sqrt{2}}{2}\eta_2\};\\
&(ii)\qquad  \|D\| < min\{2(\sqrt{2}-1)\eta_1,\frac{\sqrt{2}}{2},\frac{\sqrt{2}}{2}\eta_2\}.
\end{aligned}
\end{equation}
Then, there exist $\cP-$measurable processes 
$(X,(p^1,q^1),...,(p^m,q^m))$ such that $X$ and $p^i, i=1,\dots, m$ belong to $\mathcal{S}^{2,n}$ and $q^i, i=1,\dots, m$ belong to $\mathcal{H}^{2,n}$ which solve the Backward-Forward stochastic differential equation of mean-field type \eqref{eqbfsde}.
\end{proposition}
\begin{proof}
Recall the BFSDE \eqref{eqbfsde2} is  
\begin{eqnarray}\label{eqbfsde3}\left\{\begin{array}{l}
X_t=x+\int_0^t\{A_sX_s-\tilde Y_s+D_s\E[X_s]+\beta_s\}ds+\int_0^t\{\sigma_s X_s+\alpha_s\}dW_s\,,\\
\\
\tilde Y_t=(\sum_{k=1,m}K^iQ^i)X_T+
(\sum_{k=1,m}K^iR^i)\E[X_T]-\\ \qquad\qquad \int_t^T\{-A^\top_s\tilde Y_s-(\sum_{k=1,m}K^iM^i_s)X_s-\E[D_s^\top\tilde Y_s]-\sigma_s^\top \tilde Z_s\}ds-\int_t^T\tilde Z_sdW_s.
\end{array}\right.
\end{eqnarray}
 When (H2) holds, 
 \begin{equation}
 \begin{aligned}
  \mathcal{A}(t,u,u^{\prime},\nu)&=-|y-y'|^2-(\sum_{k=1}^mK^iM^i_s)|x-x'|^2\\
 & \leq -|y-y'|^2-\eta_2|x-x'|^2
  \end{aligned}
 \end{equation}
 which means that $k=min\{1,\eta_2\}$.
 For any $x,x^{\prime} \in \R^n, \nu\in\bM_2(\R^{n}\times\R^n)$
  \begin{equation}
 \begin{aligned} 
 g(x,\nu)-g(x^{\prime},\nu))\cdot(x-x^{\prime})=x^\top (\sum_{i=1,m}K^iQ^i)x\ge \eta_1|x|^2
  \end{aligned}
 \end{equation}
 which means that $k^{\prime}= \eta_1$. Now, under conditions \eqref{Existence and uniqueness solution of game} we can apply Theorem \ref{T123-2}, to deduce the existence of $\cP-$measurable processes $(X,\tilde Y,\tilde Z)$ which solve the MF-BFSDE\eqref{eqbfsde3}.
 \medskip
 
We will now  prove that when (ii) is satisfied, this solution is unique without using Theorem \ref{T123-2}. This is due to the fact that in this specific case, uniqueness is obtained in an easy way without strong conditions on the Lipschitz constants of $f$, $h$, $g$ and $\s$ as it is the case in Theorem \ref{T123-2}.

Assume there is another solution $(X',Y',Z')$ of  \eqref{eqbfsde3} and set
$$
\D X=X-X', \,\, \D  Y=\tilde Y-\tilde Y'\,\, \text{ and } \,\, \D Z=\tilde Z-\tilde Z'.
$$
We have, for every  $0\le t\le T$,
\begin{align}
&\nn \D X_t=\int_0^t\{A_s\D X_s-\D  Y_s+D_s\E[\D X_s]\}ds+\int_0^t\sigma_s \D X_sdW_s
\\&\nn 
\D  Y_t=(\sum_{k=1}^mK^iQ^i)\D X_T+
(\sum_{k=1}^mK^iR^i)\E[\D X_T]-\\&\nn \qquad\qquad \int_t^T\{-A^\top_s\D  Y_s-(\sum_{k=1}^mK^iM^i_s)\D X_s-\E[D_s^\top \D Y_s]-\sigma_s^\top  \D Z_s\}ds-\int_t^T \D Z_sdW_s.
\end{align}
Next, applying It\^o's formula to $\D X^\top\D Y$ and taking expectation we obtain
\begin{align}
\E[\D X_T^\top \D Y_T]&=\E\{
\D X_T^\top(\sum_{k=1}^mK^iQ^i)\D X_T+
\E[\D X_T]^\top(\sum_{k=1}^mK^iR^i)\E[\D X_T]\} \nn \\&=\E[\int_0^T\{-|\D Y_s|^2-
\D X_s^\top(\sum_{k=1}^mK^iM^i_s)\D X_s\}ds].
\end{align}
This gives
\begin{align*}
\E[
\eta_1\D X_T^\top\D X_T]-C^\nu_g
\E[\D X_T]^\top\E[\D X_T]\leq \E[\int_0^T\{-|\D Y_s|^2-
\D X_s^\top(\sum_{k=1}^mK^iM^i_s)\D X_s\}ds].
\end{align*}
Thus, since $C^\nu_g=\|(\sum_{k=1}^mK^iR^i)\|< \eta_1$, by continuity of the processes we obtain
$$
\bP\mbox{-a.s.}, \,\,\, \fr t\leq T, \,\,\, X_t=X'_t \,\,\tx { and }\,\, Y_t=Y'_t
$$
and finally $Z'_t=Z_t$, $dt\otimes d\bP$-a.e. Thus, the solution of \eqref{eqbfsde3} is unique. 

Next, by the results of \cite{buckdahnlipeng}, for $i=1,\ldots,m$, there exists $(p^i,q^i)\in \mathcal{S}^{2,n}\times \mathcal{H}^{2,n}$ solution of the following standard BSDE: $\bP$-a.s., $\forall t\leq T$,
$$\begin{array}{l}
p^i_t=Q^iX_T+R^i\E[X_T]+\int_t^T\{A^\top_sp^i_s+M^i_sX_s+\E[D^\top_sp^i_s]+\sigma_s^\top q^i_s\}ds-\int_t^Tq^i_sdW_s.
\end{array}$$
Therefore, the process $(X,Y=\sum_{i=1}^mK^ip^i,Z=\sum_{i=1}^mK^iq^i)$ is a solution of \eqref{eqbfsde3}. As the solution of this latter is unique, it holds that  $\tilde Y=\sum_{i=1}^mK^ip^i$ and $\tilde Z=\sum_{i=1}^mK^iq^i$. Replace now 
$\tilde Y$ (resp. $\tilde Z$) with $\sum_{i=1}^mK^ip^i$ (resp. $\sum_{i=1}^mK^iq^i$) in \eqref{eqbfsde2} to obtain that $(X,(p^i,q^i)_{1\le i\le m})$ satisfy the FBSDE \eqref{eqbfsde}. The proof is complete. 
\end{proof}
As an immediate consequence of Propositions \ref{prop33} and \ref{Existence and uniqueness solution of game}, we give the main result of this section.
\begin{theorem} \lb{thm35}Assume that (H2) holds and  the following conditions are satisfied:
\begin{equation}\label{condition of game}
\begin{aligned}
&(i)\,\,\|(\sum_{k=1}^mK^iR^i)\| < min\{2(\sqrt{2}-1)\eta_1,\frac{\sqrt{2}}{2},\frac{\sqrt{2}}{2}\eta_2\};\\
&(ii)\,\,\|D\| < min\{2(\sqrt{2}-1)\eta_1,\frac{\sqrt{2}}{2},\frac{\sqrt{2}}{2}\eta_2\}.
\end{aligned}
\end{equation}
Then, the collective strategy $\tilde u=((-(N^i)^{-1}(C^i)^\top p_i(t))_{t\leq T})_{1\le i\le m}$, where  $(X,(p^i,q^i)_{1\le i\le m})$ is the solution of FBSDE \eqref{eqbfsde}, is a Nash eqilibrium point for the mean-field LQ differential game.
\end{theorem}
\bs

\begin{example}[Nonexistence of a Nash Equilibrium Point of specific game problem without condition \eqref{condition of game}]${}$ \\

 We give an example to illustrate that when \eqref{condition of game} is not satisfied, the game may not have an equilibrium point. The idea is inspired by the conclusion shown in Section 6 of  \cite{eisele1982nonexistence} and Example (4.b) in \cite{hamadene1999nonzero}.\\ \\ Consider following game problem:
\begin{equation}\label{example}
\begin{aligned}
dX_t&=\left\{X_t-\E[X_t]+u(t)\left[\begin{array}{ccc}
1\\
-2
\end{array}\right]+v(t)\left[\begin{array}{ccc}
-2\\
1
\end{array}\right]\right\}dt+dW_t,\,\,t\leq T ;X_0=(1 ,\,2)^\top.
\end{aligned}
\end{equation}
Let $J_1$ and $J_2$ be the cost functionals defined by: $$
J_1(u,v)=\frac{1}{2}\E[\int_0^T (u(t))^2 dt +(X_T^1)^2]\mbox{ and } J_2(u,v)=\frac{1}{2}\E[\int_0^T (v(t))^2 dt +(X_T^2)^2]$$
where $u, v$ are $\R$-valued and $\F_t$- adapted process. Here, the associated $D=-1$ and then $\|D\|=1$, which does not satisfy \eqref{condition of game}-(ii).  The Mean-Field FBSDE associated  with the game is, for every $t\le T$,
\begin{equation}\label{FBSDE of example}
\left\{
\begin{aligned}
X_t&=\left[\begin{array}{ccc}
1\\
2
\end{array}\right]+\int_0^t\{X_s-\E[X_s]+\left[\begin{array}{ccc}
1\\
-2
\end{array}\right][-1\ \ 2]p^1_s+\left[\begin{array}{ccc}
-2\\
1
\end{array}\right][\ 2\ \ -1]p^2_s\}ds+W_t;\\
p^1_t&=\left[\begin{array}{ccc}
X^1_T\\
0
\end{array}\right]+\int_t^T\{p^1_s-\E[p^1_s]\}ds-\int_t^Tq^1_sdW_s;\\
p^2_t&=\left[\begin{array}{ccc}
0\\
X^2_T
\end{array}\right]+\int_t^T\{p^2_s-\E[p^2_s]\}ds-\int_t^Tq^2_sdW_s.
\end{aligned}
\right.
\end{equation}
Next, set $Y=(\E(X_t))_{t \leq T}$, $\overline{p}^i=(\E [p^i(t)])_{t\leq T}, i=1,2$ . Taking expectation in \eqref{FBSDE of example}, we obtain, for every $t\le T$, 
\begin{equation}\label{FBODE of example}
\left\{
\begin{aligned}
Y_t&=\left[\begin{array}{ccc}
1\\
2
\end{array}\right]+\int_0^t\left\{\left[\begin{array}{ccc}
1\\
-2
\end{array}\right][-1\ \ 2]\overline{p}^1_s+
\left[\begin{array}{ccc}
-2\\
1
\end{array}\right][\ 2\ \ -1]\overline{p}^2_s\right\}ds;\\\overline{p}^1_t&=\left[\begin{array}{ccc}
Y^1_T\\
0
\end{array}\right],
\overline{p}^2_t=\left[\begin{array}{ccc}
0\\
Y^2_T
\end{array}\right].
\end{aligned}
\right.
\end{equation}
which is a deterministic system. With the previous system is associated the following deterministic nonzero-sum game
\begin{equation}\label{deter game of example}
\left\{
\begin{aligned}
dY_t&=\left\{\overline{u}(t)\left[\begin{array}{ccc}
1\\
-2
\end{array}\right]+\overline{v}(t)\left[\begin{array}{ccc}
-2\\
1
\end{array}\right]\right\}dt,\,\,\,t\le T,\\
Y_0&=\left[\begin{array}{ccc}
1\\
2
\end{array}\right],
\end{aligned}
\right.
\end{equation}
and the cost functionals are given by
\begin{equation}\label{dge2}\bar J_1(\overline{u},\overline{v})=\frac{1}{2}\{\int_0^T (\overline{u}(t))^2 dt +(Y_T^1)^2\}\mbox{ and } \bar J_2(\overline{u},\overline{v})=\frac{1}{2}\{\int_0^T (\overline{v}(t))^2 dt +(Y_T^2)^2\}.\end{equation}

The problem \eqref{deter game of example}-\eqref{dge2} is a deterministic nonzero-sum game. Noting that if the game problem \eqref{example} has a Nash equilibrium point, by Proposition \ref{prop33}, the MF-BFSDE \eqref{FBSDE of example} has a solution. Hence, obviously the FBODE \eqref{FBODE of example} has a solution, which means that the  deterministic game problem \eqref{deter game of example}-\eqref{dge2} has a Nash  equilibrium point. However, when we choose $T=1$, following the conclusion in \cite{eisele1982nonexistence}, the game \eqref{deter game of example}-\eqref{dge2} does not have a Nash equilibrium point and then the equation \eqref{FBODE of example} does not have a solution. Therefore, the MF-BFSDE \eqref{FBSDE of example} does not have a solution for $T=1$, from which we deduce that the game \eqref{example} does not have a Nash equilibrium point.
\end{example}


\begin{bibdiv}
\begin{biblist}

\bib{AD}{article}{
title={A maximum principle for SDEs of mean-field type},
  author={Andersson, Daniel},
  author={Djehiche, Boualem},
  journal={Applied Mathematics \& Optimization},
  volume={63},
  number={3},
  pages={341--356},
  year={2011},
  publisher={Springer}
}
\bib{antonelli}{article}{
title={Backward-forward stochastic differential equations},
  author={Antonelli, Fabio},
  journal={The Annals of Applied Probability},
  volume={3},
  number={3},
  pages={777--793},
  year={1993},
  publisher={Institute of Mathematical Statistics}
}

\bib{bensoussan81}{article}{
title={Lectures in stochastic control},
author={Bensoussan, Alain},
journal={Proc. Cortona 1981, Lecture Notes in Mathematics},
volume={972},
pages={1--162},
year={1982},
publisher={Springer}
}

\bib{buckdahnlipeng}{article}{
  title={Mean-field backward stochastic differential equations and related partial differential equations},
  author={Buckdahn, Rainer},
  author={Li, Juan},
  author={Peng, Shige},
  journal={Stochastic Processes and their Applications},
  volume={119},
  number={10},
  pages={3133--3154},
  year={2009},
  publisher={Elsevier}
}

\bib{BDL}{article}{
title={A general stochastic maximum principle for SDEs of mean-field type},
  author={Buckdahn, Rainer},
  author={ Djehiche, Boualem}
  author={Li, Juan},
  journal={Applied Mathematics \& Optimization},
  volume={64},
  number={2},
  pages={197--216},
  year={2011},
  publisher={Springer}
}



\bib{BLM}{article}{
title={A stochastic maximum principle for general mean-field systems},
  author={Buckdahn, Rainer},
  author={Li, Juan},
  author={Ma, Jin},
  journal={Applied Mathematics \& Optimization},
  volume={74},
  number={3},
  pages={507--534},
  year={2016},
  publisher={Springer}
}



\bib{cadelinaskaratzas95}{article}{
title={The stochastic maximum principle for linear, convex optimal control with random coefficients},
  author={Cadenillas, Abel},
  author={Karatzas, Ioannis},
  journal={SIAM journal on control and optimization},
  volume={33},
  number={2},
  pages={590--624},
  year={1995},
  publisher={SIAM}
}

\bib{carmona2013}{article}{
  title={Mean field forward-backward stochastic differential equations},
  author={Carmona, Ren{\'e} },
  author={Delarue, Fran{\c{c}}ois},
  journal={Electronic Communications in Probability},
  volume={18},
  year={2013},
  publisher={The Institute of Mathematical Statistics and the Bernoulli Society}
}

\bib{carmona2015}{article}{
  title={Forward--backward stochastic differential equations and controlled McKean--Vlasov dynamics},
  author={Carmona, Ren{\'e} },
  author={Delarue, Fran{\c{c}}ois},
 journal={The Annals of Probability},
  volume={43},
  number={5},
  pages={2647--2700},
  year={2015},
  publisher={Institute of Mathematical Statistics}
}
\bib{DH}{article}{
  title={Optimal control and zero-sum stochastic differential game problems of mean-field type},
  author={ Djehiche, Boualem  }
  author={ Hamad{\`e}ne,Said },
  journal={Applied Mathematics \& Optimization},
  pages={1--28},
  year={2018},
  publisher={Springer}
}
 \bib{duncan-tembine}{article}{
  title={Linear-quadratic mean-field-type games: A direct method.},
  author={T.E. Duncan and H.Tembine},
  journal={Games},
  volume={9},
  number={1},
  pages={p.7},
  year={2018},
  publisher={Springer}
}

\bib{eisele1982nonexistence}{article}{
  title={Nonexistence and nonuniqueness of open-loop equilibria in linear-quadratic differential games},
  author={Eisele, T.},
  journal={Journal of Optimization Theory and Applications},
  volume={37},
  number={4},
  pages={443--468},
  year={1982},
  publisher={Springer}
}
 \bib{hamadene1999nonzero}{article}{
  title={Nonzero sum linear--quadratic stochastic differential games and backward--forward equations},
  author={Hamad\`ene, Sa{\i}d},
  journal={Stochastic Analysis and Applications},
  volume={17},
  number={1},
  pages={117--130},
  year={1999},
  publisher={Taylor \& Francis}
}
 
\bib{H98}{article}{
  title={Backward--forward SDEs and stochastic differential games},
  author={Hamad{\`e}ne, Sa{\i}d},
  journal={Stochastic processes and their applications},
  volume={77},
  number={1},
  pages={1--15},
  year={1998},
  publisher={Elsevier}
}
\bib{hu1995solution}{article}{
  title={Solution of forward-backward stochastic differential equations},
  author={Hu, Ying }
  author={ Peng, Shige},
  journal={Probability Theory and Related Fields},
  volume={103},
  number={2},
  pages={273--283},
  year={1995},
  publisher={Springer}
}   

\bib{hu-yong}{article}{
  title={Forward--backward stochastic differential equations with nonsmooth coefficients},
  author={Hu, Ying}
  author={ Yong, Jiongmin},
  journal={Stochastic processes and their applications},
  volume={87},
  number={1},
  pages={93--106},
  year={2000},
  publisher={Elsevier}
}
\bib{hu-peng95}{article}{
  title={Solution of forward-backward stochastic differential equations},
  author={Hu, Ying},
  author={Peng, Shige},
  journal={Probability Theory and Related Fields},
  volume={103},
  number={2},
  pages={273--283},
  year={1995},
  publisher={Springer}
  }
  \bib{Karatzas-Shreve}{book}{
   author={Karatzas, Ioannis},
   author={Shreve, Steven},
   title={Brownian motion and stochastic calculus},
     volume={113},
   edition={2},
   publisher={Springer Science \& Business Media},
   date={2012},
   }
  
  \bib{MaP}{article}{
  title={Solving forward-backward stochastic differential equations explicitly: a four step scheme},
  author={Ma, Jin},
  author={Protter, Philip},
  author={Yong, Jiongmin},
  journal={Probability theory and related fields},
  volume={98},
  number={3},
  pages={339--359},
  year={1994},
  publisher={Springer}
  
}

\bib{MP}{article}{
   author={ Miller, Enzo},
   author={Pham, Huyen},
   title={Linear-Quadratic McKean-Vlasov Stochastic Differential Games },
     booktitle={Modeling, Stochastic Control, Optimization, and Applications},
  pages={451--481},
  year={2019},
  publisher={Springer}  
}
\bib{Min}{article}{
  title={Fully coupled mean-field forward-backward stochastic differential equations and stochastic maximum principle},
  author={Min, Hui},
  author={Peng, Ying},
  author={ Qin, Yongli},
  booktitle={Abstract and Applied Analysis},
  volume={2014},
  year={2014},
  organization={Hindawi}
}


   
   \bib{Peng-wu}{article}{
  title={Fully coupled forward-backward stochastic differential equations and applications to optimal control},
  author={Peng, Shige},
  author={Wu, Zhen},
  journal={SIAM Journal on Control and Optimization},
  volume={37},
  number={3},
  pages={825--843},
  year={1999},
  publisher={SIAM}
}


\end{biblist}
\end{bibdiv}
\end{document}